\documentclass[reqno,english]{amsart}
\usepackage[foot]{amsaddr}

\numberwithin{equation}{section}

\usepackage{amsmath,amsfonts,amssymb,graphicx,amsthm,enumerate,url}
\usepackage[noadjust]{cite}
\usepackage{stmaryrd}
\usepackage{comment,paralist}
\usepackage{mathrsfs,booktabs,tabularx}
\usepackage{xifthen,xcolor,tikz,setspace}
\usetikzlibrary{decorations.pathmorphing,patterns,shapes,calc,decorations}
\usetikzlibrary{decorations.pathreplacing}
\usepackage{mathtools}

\usepackage{caption}
\usepackage{subcaption}

 \usepackage[letterpaper]{geometry}
\usepackage[colorinlistoftodos]{todonotes}
\usepackage[colorlinks=true]{hyperref}

\usepackage{amsmath}
\usepackage{upgreek}
\usepackage{amssymb}
\usepackage{bbm}
\usepackage{amsthm}
\usepackage{mathtools}
\usepackage[capitalize]{cleveref}

\newtheorem{proposition}{Proposition}[section]
\newtheorem{theorem}{Theorem}
\newtheorem{corollary}{Corollary}[section]
\newtheorem{conjecture}{Conjecture}
\newtheorem{lemma}{Lemma}[section]
\newtheorem{definition}{Definition}[section]

\newcommand*{\PP}[2][]{\mathbb{P}_{#1}\left[#2\right]}
\newcommand*{\pp}{\mathbf{P}}%
\newcommand*{\E}[2][]{\mathbb{E}_{#1\unskip\space}\left[#2\right]}
\newcommand*{\EE}{\mathbb{E}}%
\newcommand*{\cO}[1]{\mathcal{O}\left(#1\right)}
\newcommand*{\co}[1]{\oo\left(#1\right)}
\newcommand*{\cOm}[1]{\Omega\left(#1\right)}
\newcommand*{\cT}[1]{\Theta\left(#1\right)}
\newcommand*{\cE}[1]{\exp\left\{#1\right\}}
\newcommand*{\oo}{\mathrm{o}}
\newcommand*{\mO}[1]{\left(1 + \mathcal{O}\left(#1\right)\right)}
\newcommand*{\mo}[1]{\left(1 + \mathrm{o}\left(#1\right)\right)}

\newcommand*{\K}{\mathcal{K}}%
\newcommand*{\M}{\mathcal{M}}%
\newcommand{\disc}{\mathrm{disc}}

\newcommand*{\Sa}{\sigma_{\ddag}}%
\newcommand*{\Sb}{\sigma_{\pm}}%
\newcommand*{\Sc}{\sigma_=}%
\newcommand*{\Sd}{\sigma_{\mp}}%

\newcommand*{\Ta}{\Sa}%
\newcommand*{\Tb}{\Sb}%
\newcommand*{\Tc}{\Sc}%
\newcommand*{\Td}{\Sd}%
\newcommand*{\dd}{l}
\newcommand*{\Ej}{E}

\newcommand*{\e}{\mathrm{e}}
\newcommand*{\ts}{\Tilde{s}}

\newcommand*{\cEws}{\e^{-\cT{\frac{1}{w}}\left( \frac{w}{2}-s\right)^2}} %

\newcommand*{\antisym}{\Xi}%

\newcommand*{\QEDA}{\null\nobreak\hfill\ensuremath{\square}}
\newcommand*{\ZZ}{\mathbb{Z}}%
\newcommand*{\RR}{\mathbb{R}}%
\newcommand*{\NN}{\mathbb{N}}%
\newcommand*{\WW}{\mathcal{W}}%
\newcommand*{\ind}{\mathbbm{1}}%
\newcommand*{\la}{\langle}%
\newcommand*{\ra}{\rangle}%
\newcommand*{\eps}{\epsilon}%
\newcommand*{\del}{\delta}%

\newcommand{\Pois}{\mathrm{Po}}
\newcommand{\Bin}{\mathrm{Bin}}

\newcommand{\defeq}{\coloneqq}

\let\originalleft\left
\let\originalright\right
\def\left#1{\mathopen{}\originalleft#1}
\def\right#1{\originalright#1\mathclose{}}

\title{The Discrepancy of Random Rectangular Matrices}
\author{Dylan J. Altschuler}
\address{D.J.\ Altschuler\hfill\break
Courant Institute\\ New York University\\
251 Mercer Street\\ New York, NY 10012, USA.}
\email{dylan.altschuler@courant.nyu.edu}
\author{Jonathan Niles-Weed}
\address{J. \ Niles-Weed \hfill\break
Courant Institute\\ New York University\\
251 Mercer Street\\ New York, NY 10012, USA.}
\email{jnw@cims.nyu.edu}
\thanks{DJA and JNW were supported in part by NSF grant DMS-2015291. DJA was supported in part by NSF Graduate Research Fellowship Program grant DGE-1839302.}

\begin{document}

\maketitle
\begin{abstract}
A recent approach to the Beck--Fiala conjecture, a fundamental problem in combinatorics, has been to understand when random integer matrices have constant discrepancy. We give a complete answer to this question for two natural models: matrices with Bernoulli or Poisson entries. For Poisson matrices, we further characterize the discrepancy for any rectangular aspect ratio. These results give sharp answers to questions of Hoberg and Rothvo\ss~(SODA 2019) and Franks and Saks (\emph{Random Structures Algorithms} 2020). Our main tool is a conditional second moment method combined with Stein's method of exchangeable pairs. While previous approaches are limited to dense matrices, our techniques allow us to work with matrices of all densities. This may be of independent interest for other sparse random constraint satisfaction problems.


\end{abstract}

\section{Introduction}
Given a matrix $A \in \RR^{m \times n}$, its \emph{discrepancy} is
\begin{equation}\label{eq:discrepancy}
\disc(A) \defeq \min_{u \in \{-1, +1\}^n} \|A u\|_{\infty}\,.
\end{equation}
If $A$ has entries in $\{0,1\}$, we can view $A$ as the incidence matrix of a collection of $m$ subsets of $[n]$. Then this optimization problem asks to partition $[n]$ into two classes so that the imbalance of each of the $m$ sets is small.
Discrepancy measures how well this partitioning can be done.
This problem, a natural generalization of graph coloring, has been studied widely in combinatorics and computer science, due to its connections to problems such as integer rounding, set balancing, and metric embeddings~\cite{Cha00,Spe94}.

The foundational result in the field is Spencer's celebrated ``Six Standard Deviations Suffice''~\cite{spencer1985six}, which states that $\disc(A) \leq 6 \sqrt{n}$.
By contrast, if $m \asymp n$, a typical vector $u \in \{-1, +1\}^n$ has $\|A u\|_{\infty} = \cT{\sqrt{n \log n}}$.
Spencer's result reveals the surprising fact that it is possible to find a vector $u$ for which $\|A u\|_{\infty}$ is much smaller than this typical value.

Spencer's theorem is unimprovable in general, but it raises the question of which matrices enjoy better bounds---specifically, bounds that are independent of the matrix dimensions.
Beck and Fiala~\cite{beck1981integer} showed that if $A$ is the incidence matrix of a {\em $t$-sparse} set-system, i.e. $A$ has binary entries and columns summing to at most $t$, then $\disc(A) \leq 2t - 1$, independent of the dimensions of the matrix.
They further conjectured that this bound is improvable to $\disc(A) = \cO{\sqrt t}$.
This conjecture has resisted significant progress: to date, the best dimension-independent bound, due to Bukh, is just $\disc(A) \leq 2t - \log^* t$, where $\log^*$ is the iterated logarithm function~\cite{bukh2016beck}.
If mild dependence on the dimension is allowed, then the best bound is due to Banaszczyk~\cite{banaszczyk1998balancing}, who showed that $\disc(A) = \cO{\sqrt{t \log n}}$.
Neither approach seems likely to yield a proof of Beck and Fiala's conjectured bound.

Since the Beck-Fiala conjecture seems beyond the grasp of current techniques, there has been recent interest in understanding randomized versions of the problem.
A striking finding of this line of work is that the random setting evinces sharply different behavior in two different regimes: loosely speaking, past results show that a random binary $m \times n$ matrix $A$ has discrepancy $\cT{\sqrt n}$ when $n \asymp m$, but when $n$ is significantly larger than $m$, then a dimension-free bound is possible---in fact, $\disc(A) = 1$ with high probability.
This is the smallest possible discrepancy, since any row whose sum is odd must have discrepancy at least $1$.

Prior work has investigated this phenomenon in a variety of different parameter ranges, but understanding exactly when $\disc(A) \leq 1$ is achievable has remained an open question.
If $A$ is an $m \times n$ matrix\footnote{For notational simplicity, we restrict to even $n$ in the remainder of our paper.} with independent Bernoulli$(p)$ entries, then it is known that $\disc(A) = \cO{1}$ with high probability as long as $A$ is very wide ($n \gg m^2$) and relatively dense ($p \gg n^{-1/2}$).
In the special symmetric case where $p = 1/2$, Potukuchi \cite{potukuchi2018discrepancy} proved that $\disc(A) \leq 1$ as long as $n \geq C m \log m$ for a sufficiently large constant, but this is not known for any other values of $p$.
Taken together, these works suggest the presence of a threshold above which constant discrepancy is achievable, but they do not give a hint of where this threshold should be.

Our first main result solves this problem by identifying the precise location at which this transition occurs. Strikingly, the threshold is independent of $p$, and is valid even for $p = p(n)$ varying with $n$.
\begin{theorem}\label{main}
Let $A \in \{0,1\}^{m \times n}$ be a random matrix whose entries are i.i.d.~Bernoulli random variables with parameter $p := p(n)$.
There exists a universal constant $C > 0$ such that if $n \geq C m \log m$, then $\disc(A) \leq 1$ with high probability.
\end{theorem}
The constant in \cref{main} can be made explicit: under mild assumptions, it suffices to let $C$ be any constant strictly larger than $(2 \log 2)^{-1}$, which is precisely the threshold at which the expected number of low-discrepancy vectors becomes large. It is easy to see that this cannot be improved in general; for example, if $p = 1/2$ and $n = C m \log m$ for $C< (2\log 2)^{-1}$, then by Markov's inequality the probability that $\disc(A) = \cO{1}$ is exponentially small. 

\Cref{main} shows that $\disc(A) \leq 1$ holds when $n  \geq C m \log m$ for a sufficiently large constant $C$ irrespective of the value of $p$, substantially generalizing Potukuchi's result~\cite{potukuchi2018discrepancy}.
As we discuss in more detail below, the sparse regime where $p = \oo(1)$ evinces fundamentally different behavior from the $p = 1/2$ case, and requires different techniques.
The key challenge is that approximations based on the central limit theorem---which are valid for dense matrices---become too inaccurate when $A$ is sparse. We therefore need to develop tools to obtain precise approximations in a regime where the CLT and other classic asymptotic methods break down.

Establishing \cref{main} requires a quantitative understanding of how dependent the events $\{\|Ax\|_{\infty} \le 1\}$ and $\{\|Ay\|_{\infty} \le 1\}$ are, for pairs of vectors $x,y\in \{\pm 1\}^n$. More precisely, we need an upper-bound on 
\begin{equation}\label{eq:2mm-goal-intro}
    \PP{|Ay\|_{\infty} \le 1 \big |\|Ax\|_{\infty} \le 1} - \PP{\|Ay\|_{\infty} \le 1}
\end{equation}
When $A$ is sufficiently dense, both probabilities in \cref{eq:2mm-goal-intro} can be individually approximated to sufficient accuracy via classical tools. However, when $A$ is too sparse, this na\"ive approach no longer succeeds. To handle this difficulty, we instead directly compare the two probabilities using a version of Stein's method called {\em the method of exchangeable pairs} ~\cite{barbour1992poisson,diaconis2004stein}.

Originally developed to prove CLTs~\cite{stein1972bound}, Stein's method has proven to be a powerful general tool for establishing limit laws for dependent random variables. Informally, to compare a complicated distribution $\mu_c$ on a set $\mathcal X$ to a target distribution $\mu_0$, Stein proposed to find an operator $T_0$, acting on functions from $\mathcal X \to \RR$, which satisfies the requirement that
\begin{equation*}
\E[\mu_0]{T_0 f} = 0 \quad \forall f: \mathcal X \to \RR\,.
\end{equation*}
Then, so long as it can be shown that $\E[\mu_c]{T_0 f} \approx 0$ for all $f$ in a set of suitably rich test functions, one can conclude that the distribution $\mu_c$ is close to $\mu_0$.
Though proving $\E[\mu_c]{T_0 f} \approx 0$ by hand can be challenging, it is often possible to find an operator $T_c$, satisfying $\E[\mu_c]{T_c f} = 0$ for all $f: \mathcal X \to \RR$, such that $\|T_c f - T_0 f\|_\infty$ is small---in this case, we will have
\begin{equation*}
|\E[\mu_c]{T_0 f}| = |\E[\mu_c]{T_0 f}- \E[\mu_c]{T_c f}| \leq \|T_c f - T_0 f\|_\infty \approx 0\,,
\end{equation*}
which is the desired claim. The method of exchangeable pairs gives a simple way of constructing $T_0$ and $T_c$ from reversible Markov chains with stationary measures $\mu_0$ and $\mu_c$.

To use Stein's method, we view \cref{eq:2mm-goal-intro} as an expression measuring how different the law of $\|A y\|_\infty$ is from the law of $\|A y\|_\infty$ conditioned on $\{\|A x\|_\infty \leq 1\}$. To compare these two measures, we construct two Markov chains with the measures as stationary distributions, and use the method of exchangeable pairs to find suitable operators $T_0$ and $T_c$.
By ensuring that the two Markov chains have similar transition probabilities, we can guarantee that $T_0 - T_c$ is easy to control.
Though Stein's method is well known in the probability literature, its use in the context of the second moment method appears to be new.

\Cref{main} is stated for matrices with Bernoulli entries. In addition to studying this model, we also introduce a natural extension, about which we can prove more powerful bounds.
We formalize these two models in the following definition.
\begin{definition}[Bernoulli, Poisson Ensembles]
    Let $A$ be an $m \times n$ random matrix with independent and identically distributed entries. If $A_{ij}$ is Bernoulli$(p)$, we say $A$ is from the $(m,n,p)$-Bernoulli ensemble. If $A_{ij}$ is Poisson$(\lambda)$, we say $A$ is from the $(m,n,\lambda)$-Poisson ensemble.
\end{definition}
By symmetry, we may always assume in the Bernoulli ensemble that $p \leq 1/2$.
It is useful to view both ensembles as the adjacency matrices of random bipartite factor graphs, where the columns of $A$ correspond to the vertices and the rows correspond to the factors.
The Bernoulli ensemble corresponds to an Erd\H os-R\'enyi model, which is the one common in the recent discrepancy literature.
The Poisson ensemble is a natural extension in which multi-edges are allowed.

In the Poisson ensemble, we are able to prove a significant generalization of \cref{main} by characterizing the behavior of the discrepancy for any rectangular matrix with $m = \oo(n)$.
We first define a convenient set of candidate solutions.
\begin{definition}\label{def:balanced}
A vector $u \in \{-1, +1\}^n$ is \emph{balanced} if $\sum_{i=1}^n u_i = 0$. We write $\mathcal{B}$ for the set of balanced vectors.
For a random matrix $A$ and any $r \geq 0$, we write $Z_r$ for the random variable equal to the number of $u \in \mathcal{B}$ for which $\|A u\|_\infty \leq r$.
\end{definition}
Our second main theorem characterizes the discrepancy of Poisson matrices: a matrix from the Poisson ensemble has discrepancy at most $r$ so long as $\E{Z_r}$ is large.
\begin{theorem}\label{fullcurve}
Let $A$ be drawn from the $(m, n, \lambda$)-Poisson ensemble.
If $m = \oo(n)$ and~$\log \E{Z_r} = \cT{n}$, then~$\disc(A) \leq r + 1$ with high probability.
\end{theorem}
\Cref{fullcurve} shows that the prediction based on the annealed entropy $\log \E{Z_r}$ is correct: as soon as low discrepancy solutions exist in expectation, they exist with high probability. Note that a converse statement holds by Markov's inequality: if the expected number of solutions is vanishing, then with high probability there are no solutions.
This theorem captures the transition from constant discrepancy to $\cT{\sqrt n}$ discrepancy that occurs as $n$ ranges between $m$ and $m \log m$.
For example, if $n \ge m$ and $n\lambda = \omega(1)$, then it is straightforward to verify that $\log \E{Z_r} = \cT{n}$ as long as $r = \cOm{2^{-n/m}\sqrt{n\lambda}}$, and \cref{fullcurve} therefore guarantees that $\disc(A) = \cO{2^{-n/m}\sqrt{n\lambda} + 1}$ for such matrices with high probability.
Apart from the $+1$ term, this is the same discrepancy bound that prior work shows is achievable for Gaussian matrices with i.i.d.~$\mathcal N(0, \lambda)$ entries~\cite{turner2020balancing}.
\Cref{fullcurve} therefore implies that the Poisson ensemble has similar qualitative behavior to a corresponding Gaussian model, even though our proofs reveal that there are significant technical differences between the two settings. 

We also remark that since $\sqrt{m\lambda} = \cOm{2^{-n/m}\sqrt{n\lambda}}$, \cref{fullcurve} matches the conjectured Beck--Fiala bound by analogy, where $m\lambda$ is the \textit{average} column sparsity (see \cref{fig}, right). Returning to our original motivation of understanding when matrices have constant discrepancy, we have as an easy consequence of \cref{fullcurve} an analogue of \cref{main} for Poisson matrices.

\begin{corollary}\label{cor:poisson-constant-discrep}
    For $A$ drawn from the $(m,n,\lambda)$-Poisson ensemble with $\lambda \le \text{poly}(n)$ and $n > C m \log m$ for some universal constant $C$, then $\disc(A) = 1$ with high probability.  
\end{corollary}

The proofs of \cref{main,fullcurve} are nonconstructive and leave open the question of whether it is possible to find a vector $u$ achieving $\|A u\|_{\infty} = 1$ in polynomial time.
As suggested by Aubin et al.~\cite{aubin2019storage}, it is possible to compare our model to a planted version of the discrepancy problem where the matrix $A$ is generated from the Bernoulli ensemble conditioned on a particular vector $u$ having low discrepancy.
Though we lack a rigorous proof that this planted model is contiguous to our original model, we conjecture that the geometry of the solution space in the original model is well-captured by its planted counterpart.
It can be shown in this planted model that clusters of low-discrepancy solutions are isolated from each other, which provides heuristic evidence for the following conjecture, which we view as an attractive question for future work.
\begin{conjecture}\label{conj}
For $A$ drawn from the $(m, n, 1/2)$-Bernoulli ensemble with $n \geq C m \log m$, there is no efficient algorithm that finds a constant discrepancy solution with high probability.
\end{conjecture}

\begin{figure}
    \centering
  \begin{subfigure}[t]{.45\textwidth}
    \includegraphics[width=\textwidth]{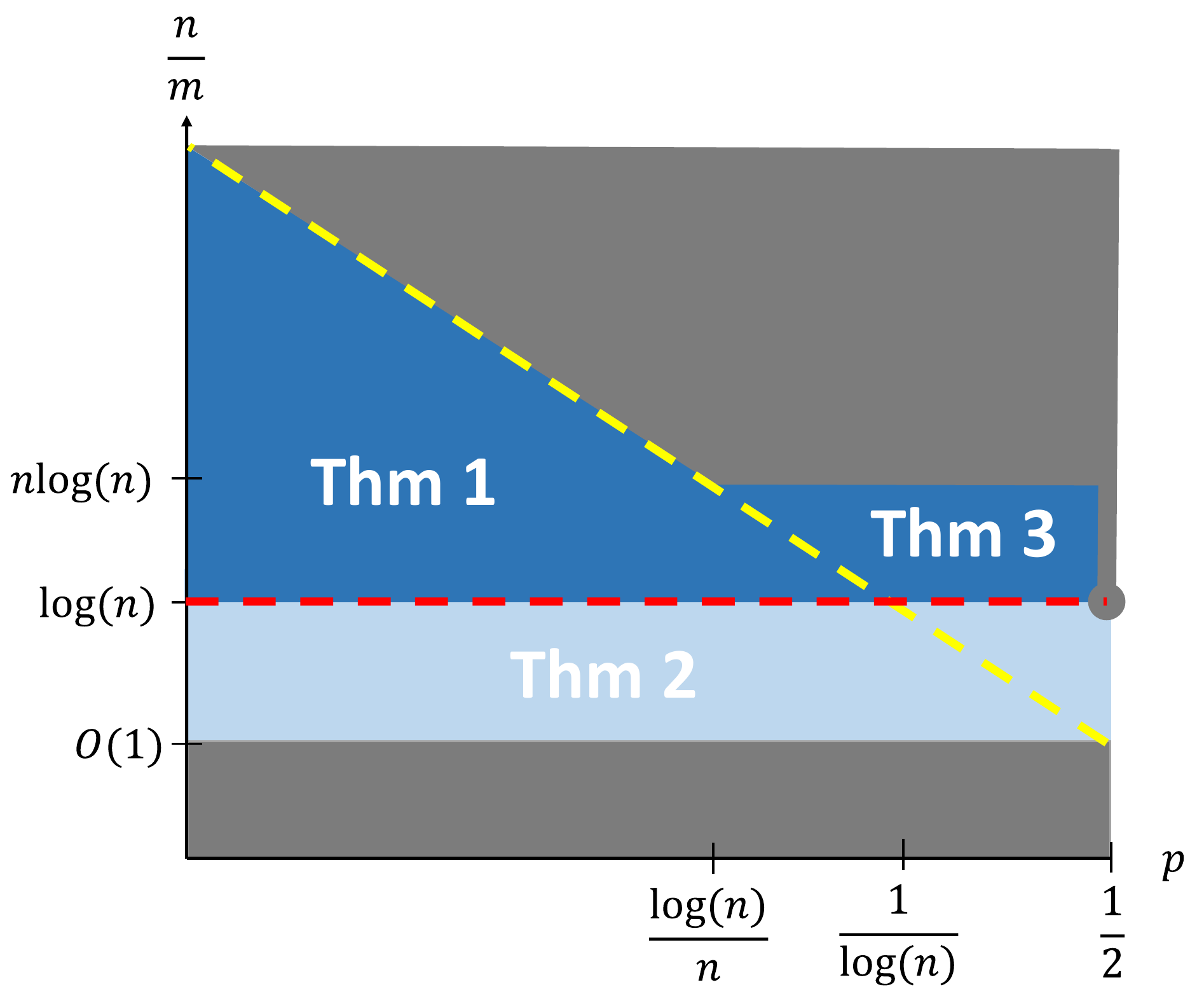}
  \end{subfigure}
  \hspace{.3cm}
  \begin{subfigure}[t]{.45\textwidth}
    \includegraphics[width=\textwidth]{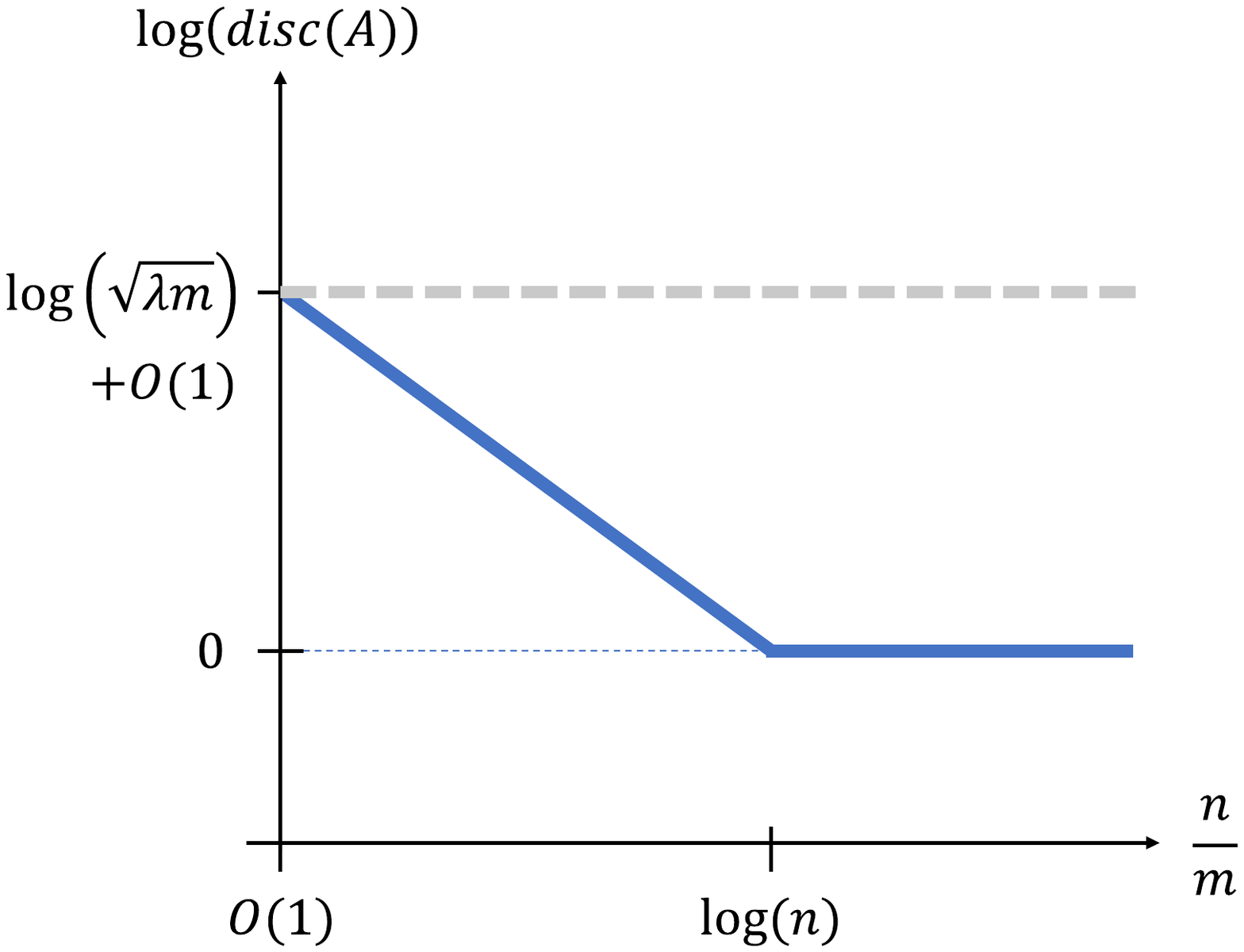}
  \end{subfigure}
  \caption{\textbf{Left}: a diagram of the state-of-the-art (grey) and our contributions (dark blue) for the discrepancy of Bernoulli random matrices. In the case of Poisson matrices (with $\lambda = p$), we contribute the light blue, as well as everything above the light blue. Above the red line, $\disc(A) = 1$ with high probability for either ensemble. Above the yellow line is the ``dense" regime. \textbf{Right}: the trade-off between aspect-ratio and discrepancy for the Poisson ensemble. The solid blue line is the sharp tradeoff for Poisson matrices given in \cref{fullcurve}. The solid grey line is conjectured to be sharp for deterministic $\lambda m$-sparse matrices (this is the ``Beck-Fiala conjecture"), and  is a known upper-bound  for random $\lambda m$-sparse matrices \cite{BanMek20}. Note the vertical axis is log-scaled.}
  \label{fig}
\end{figure}

\subsection{Previous Work}
A number of recent results study different random matrix models for which $\disc(A) = \cO{1}$ with high probability.
Ezra and Lovett \cite{ezra2019beck} consider a regular model in which a binary matrix $A$ is chosen uniformly at random conditioned on each column having exactly $t$ ones for some sparsity parameter $t$. They show when $n > m^t$, then $\disc(A) = \cO{1}$ with high probability. 

Two independent and concurrent works removed this exponential dependence on $t$. Franks and Saks \cite{franks2020discrepancy} consider a fairly general class of matrices and show $\disc(A) \le 2$ with high probability if $n = \Omega \left(m^3 \log^2 m\right)$. Simultaneously, Hoberg and Rothvo\ss~\cite{hoberg2019fourier} consider $A$ drawn from the $(m,n,p)$-Bernoulli Ensemble and give the improved bound that $\disc(A) \le 1$ with high probability if $n =\Omega \left(m^2 \log m\right)$ and $mp = \Omega(\log n)$. Shortly after, Potukuchi \cite{potukuchi2018discrepancy} improved this to $\disc(A) \le 1$ if $n = \Omega(m \log m)$ in the special symmetric case of $p = 1/2$. Both \cite{franks2020discrepancy} and \cite{hoberg2019fourier} used Fourier methods, while \cite{potukuchi2018discrepancy} used the second moment method. Around the same time as our paper, Macrury et al. \cite{macrury2021phase} showed $\disc(A) \ge \Omega(2^{-n/m}\sqrt{np})$ when $m/(np) \to 0$, for all $n = \Omega(m)$, via a first moment computation. This is analogous to our \cref{eq:fmm_heuristic}, but for a wider parameter range.

Bansal and Meka \cite{BanMek20} also improved upon \cite{ezra2019beck}, except with focus on the Beck-Fiala bound rather than constant discrepancy. They prove that, under a mild growth condition on $t$, random $m\times n$ binary matrices with $t$ ones per column have $\disc(A) \le \cO{\sqrt{t}}$ with high probability. We obtain similar results for Poisson matrices with average column weight $t$---See the discussion following \cref{fullcurve} comparing the tight rate for random Poisson matrices to the Beck-Fiala bound.

Importantly, all previous work on constant discrepancy on the Bernoulli ensemble requires
\begin{equation}\label{eq:mnp}
    \lim_{m \to \infty}\frac{m}{np} \to 0\,.
\end{equation}
We call this choice of parameters the {\em dense} regime.
There are {\em a priori} reasons to expect a non-trivial phase transition when $m/(np) \not\to 0$; as we show in~\cref{mnp}, this threshold is the point above which the number of optimal solutions to~\eqref{eq:discrepancy} no longer concentrates sufficiently well around its expectation.
Similar phenomena appear in the analysis of random graphs, whose behavior is very different in the sparse case.
Our main technical challenge is proving constant-discrepancy results in the regime where~\eqref{eq:mnp} does not hold. \\

A separate line of research has focused on the case where $A$ is a random matrix with independent Gaussian entries~\cite{turner2020balancing, aubin2019storage,chandrasekaran2011discrepancy}, showing that $\disc(A) \asymp 2^{-n/m} \sqrt n $ when $n = \Omega(m)$; in particular, that $\disc(A) = \cO{1}$ once $n \geq C m \log m$ for $C > \frac{1}{2 \log 2}$.
Discrepancy specialized to the case of iid Gaussian entries and $n = \cT{m}$ can be seen as a symmetrized version of a famous model in statistical physics known as the binary perceptron model~\cite{aubin2019storage, talagrand1999intersecting}, the rigorous understanding of which is an area of active research~\cite{ding2019capacity}. In independent and concurrent works, Perkins and Xu \cite{perkins2021frozen} and Abbe, Li, and Sly \cite{abbe2021proof} establish the ``frozen 1-RSB" geometry of typical solutions in this setting. Further, \cite{perkins2021frozen} shows exponential concentration of the number of solutions, while \cite{abbe2021proof} gives an explicit description of the asymptotic distribution of the number of solutions as well as a proof of the ``contiguity conjecture" (namely that the planted model and null model are contiguous---see \cite{aubin2019storage,abbe2021proof} for relevant definitions and discussion).

In the Gaussian case---and more generally, for distributions with sufficiently smooth densities---optimal bounds on the discrepancy can be achieved by a direct application of the second moment method.
However, in the sparse, discrete ensembles we consider, the situation is considerably more delicate, and this approach fails.
Nevertheless, our results validate the view that, despite being significantly less well behaved than Gaussian matrices, matrices with Bernoulli or Poisson entries also have small discrepancy as soon as $n \geq C m \log m$.

For square or close-to-square matrices, a variety of efficient algorithms have been discovered matching Spencer's and Banaszczyk's bounds \cite{bansal2019algorithm, eldan2018efficient, rothvoss2017constructive, bansal2018gram, nikolov2014approximating}.
In the Beck--Fiala setting, Potukuchi~\cite{potukuchi2018discrepancy,potukuchi2019spectral} gives an efficient algorithm achieving $O(\sqrt{t})$ discrepancy for random matrices with $t$-sparse columns for any $t := t(n)$, as long as $m \ge n$. However, these approaches do not appear to extend to the constant-discrepancy regime when $m = \oo(n)$.
As Hoberg and Rothvo\ss~\cite{hoberg2019fourier} note, the lack of efficient algorithms for this regime is a common feature of combinatorial problems for which solutions are shown to exist by probabilistic means~\cite{KupLovPel12}.

The use of the second-moment method for random constraint satisfaction problems was popularized by Achlioptas and Moore~\cite{achlioptas2002asymptotic} and Frieze and Wormald~\cite{FriWor05}.
It has since been successfully applied to a diverse set of problems in theoretical computer science and combinatorics~\cite{achlioptas2006random,achlioptas2003threshold,AchNaoPer07,AchNao05}.

\subsection{Heuristics from the first and second moment}
Let us first give a heuristic justification for the fact that constant discrepancy is achievable once $n \geq C m \log m$ for $C$ large enough.
Denote by $\mathcal B$ the set of balanced vectors and let $Z= Z_1$ be the number of $u \in \mathcal B$ for which $\|A u \|_{\infty} \leq 1$.
If $Z > 0$, then $\disc(A) \leq 1$. 

Since each of the vectors in $\mathcal B$ has an equal probability of satisfying this requirement, fix some $u \in \mathcal B$. Then,
\begin{equation*}
\EE Z = \sum_{u \in \mathcal B} \PP{\|A u \|_{\infty} \leq 1} = \binom{n}{n/2} \PP{\|A u\|_\infty \leq 1}\,.
\end{equation*}
The entries of $A u$ are independent, and each is a sum of $n$ independent random variables with variance $p(1-p)$, so the local central limit theorem suggests that
\begin{equation*}
\PP{\|A u\|_{\infty} \leq 1} = \left(\frac{3}{\sqrt{2 \pi n p(1-p)}}\mo{1}\originalright)^m \,.
\end{equation*}
If $n \gg m$, we therefore expect that
\begin{equation}\label{eq:fmm_heuristic}
\EE Z \approx \exp\left(n \log 2 - \frac m2 \log np + \oo(n)\right)
\end{equation}
So long as $n \geq (1+\varepsilon) \frac{m}{2 \log 2} \log m$, this quantity is exponentially large.
In expectation, therefore, $n \geq C m \log m$ is the right scaling.

Though the annealed entropy $\log \E Z$ predicts a threshold at $n \asymp m \log m$, showing that $Z$ is indeed large with high probability past this threshold requires controlling the fluctuations of $Z$.
The classic approach is the so-called \emph{second-moment method}, based on the Paley--Zygmund inequality, which says that a nonnegative integer-valued random variable is positive with high probability as long as $\EE[Z^2] = \mo{1}\EE[Z]^2$ and $\EE[Z] > 0$.
There is a slight obstruction to naively applying the second moment method to show that the discrepancy of a random matrix is at most $1$: the second moment is skewed by the fact that conditioning on $(Au)_i \in \{-1, 0, 1\}$ biases the parity of the $i$th row of $A$, because there is only one even number in this set. There are two workarounds: one can ask for $(Au)_i \in \{-2,-1,1,2\}$ and prove that $\disc(A) \leq 2$. Or, one can condition on the event that each row of $A$ has even parity, and ask that $Au = 0$.
We adopt the second approach and use the following basic construction to extend our result to the unconditioned case.
\begin{lemma}\label{parity_coupling}
Let $A$ be from either the Bernoulli or Poisson ensemble, and let $A'$ be from the same ensemble conditioned on the sum of the entries in each row being even.
There exists a coupling of $A$ and $A'$ such that
\begin{equation*}
\disc(A) \leq \disc(A') + 1 \quad \text{a.s.}
\end{equation*}
\end{lemma}
The proof of \cref{parity_coupling} appears in \cref{lemmas}.
This conditioning approach was employed by Potukuchi~\cite{potukuchi2018discrepancy} to establish  a version of \cref{main} via the second moment method when $p = 1/2$. We show that this argument can be extended to prove constant discrepancy so long as the matrix satisfies the density requirement~\eqref{eq:mnp}. Moreover, the following theorem shows that, in fact, the density requirement is necessary, and this application of the second moment method provably fails when the matrix is too sparse.
Recall $Z_r$ is the number of balanced vectors $u \in \{-1,1\}^n$ with $\|Au\|_{\infty} \le r$.
\begin{theorem}[Dense regime]\label{mnp}
    Let $A$ be drawn from the $(m,n,p)$-Bernoulli Ensemble. Define $P$ as the event that each row of $A$ sums to an even number. There exists a universal constant $C > 0$ such that for any $n \geq C m \log m$ and $np = \omega(1)$,
    \begin{equation*}
        \frac{\E{Z_0^2|P}}{\E{Z_0|P}^2} = 
        \begin{cases}
            1 + \oo(1), &\quad m = \oo(np) \\
            \cE{\cOm{\frac{m}{np}}}, &\quad m = \Omega(np)
        \end{cases}
    \end{equation*}
\end{theorem}
\Cref{parity_coupling} immediately yields the following corollary.
\begin{corollary}
Let $A$ be drawn from the $(m,n,p)$-Bernoulli ensemble with $n \geq Cm\log m$ and $m = \oo(np)$. With high probability, $\disc(A) = 1$.
\end{corollary}
The constant $C$ is the same constant that appears in \cref{main}, and this suffices to establish \cref{main} in the dense regime.
However, \cref{mnp} also shows that this strategy fails when $A$ is sparse.
Nevertheless, \cref{main} maintains that the prediction implied by~\eqref{eq:fmm_heuristic} is correct even when the second-moment method fails.

\subsection{Our techniques}
To prove \cref{main} in the case where~\eqref{eq:mnp} does not hold, we employ two strategies.
The failure of the second-moment calculation in~\cref{mnp} stems from the fact that, when $m/np \not\to 0$, the second moment $\EE[Z^2]$ is too sensitive to the sum of the entries in each row of $A$ when $A$ is sparse. We therefore carry out the second-moment method conditional on the weights of each row of $A$.
This technique is common in the literature, and bounding these conditional second-moments still suffices to show that $Z > 0$ with high probability~\cite{Jan96}.

However, even after conditioning, bounding the second moment requires significant care.
For notational simplicity, let us ignore the conditioning argument for now and consider the random variable $Z = Z_1$ counting the $u \in \mathcal{B}$ with $\|Au\|_{\infty} \le 1$, as before.
Since the entries of $Au$ are i.i.d., we obtain 
\begin{align*}
\EE[Z^2] & = \sum_{u, v \in \mathcal{B}} \PP{\|A u\|_\infty \leq 1, \|A v\|_\infty \leq 1} \\
& = \sum_{u, v \in \mathcal{B}} (\PP{|(A u)_1| \leq 1, |(A v)_1| \leq 1})^m \\
& = \sum_{u, v \in \mathcal{B}} \PP{|(A u)_1| \leq 1 \big | |(A v)_1| \leq 1}^m \cdot \PP{|(A v)_1| \leq 1}^m\,,
\end{align*}
and likewise,
\begin{equation*}
\EE[Z]^2 = \sum_{u, v \in \mathcal{B}} \PP{|(A u)_1| \leq 1}^m \cdot \PP{|(A v)_1| \leq 1}^m\,.
\end{equation*}
To show that $\EE[Z^2] = \mo{1} \EE[Z]^2$, we need to show that for a typical pair $u, v \in \mathcal{B}$, the events $\{|(A u)_1| \leq 1\}$ and $\{|(A v)_1| \leq 1\}$ are approximately independent, so that
\begin{equation*}
\PP{|(A u)_1| \leq 1 \big| |(A v)_1| \leq 1}^m  \approx \PP{|(A u)_1| \leq 1}^m
\end{equation*}
Proving this fact requires approximations on $\PP{|(A u)_1| \leq 1 \big| |(A v)_1| \leq 1}$ which are accurate to $1 + \oo(m^{-1})$.
However, calculating $\PP{|(A u)_1| \leq 1 \big| |(A v)_1| \leq 1}$ explicitly is infeasible; moreover, the local central limit theorem and other classical approximation techniques yield estimates which are accurate only up to a multiplicative factor of  $1 + \cO{\frac 1 {np}}$.
When $m/np \not \to 0$, these errors are unacceptably large.

Our second strategy bypasses this difficulty by employing Stein's method.
Though this method is well known in the probability literature for its utility in proving limit theorems, to our knowledge the use of this technique combined with the second-moment method is novel.
To evaluate $\PP{|(A u)_1| \leq 1 \big| |(A v)_1| \leq 1}^m$, we construct a pair of Markov chains, one of which has stationary distribution given by the law of $(Au)_1$ conditioned on the event $\{|(A v)_1| \leq 1\}$, and the other of which has stationary distribution given by the law of $(Au)_1$ \emph{without conditioning}.
Stein's method gives a means for comparing these two stationary distributions by inverting a particular functional equation involving the generators of these two chains, which allows us to approximate $\PP{|(A u)_1| \leq 1 \big| |(A v)_1| \leq 1}$ by a simpler, unconditional probability.
The resulting approximation has much smaller errors---of order $1 + \cO{\frac 1 {n}}$---and this improvement is crucial to obtaining accurate bounds when~\eqref{eq:mnp} fails.

\subsection{Notation}
The asymptotic notation $\oo(\cdot)$, $\cO{\cdot}$, $\cOm{\cdot}$, and $\cT{\cdot}$ refers to the $m \to \infty$ and therefore $n \to \infty$ limit.
Given a sequence $a = a(m)$ and a nonnegative sequence $b = b(m)$ depending on $m$, we write $a = \cO{b}$ or $b = \cT{a}$ if $|a(m)| \leq C b(m)$ for all $m$ sufficiently large
We write $a = \oo(b)$ if $\lim_{m \to \infty} \frac{|a(m)|}{b(m)} = 0$.
Unless otherwise specified, the implicit constants in these expressions are universal.
The phrase ``with high probability'' means that a sequence of events occurs with probability $1 - \oo(1)$ in this asymptotic limit. The symbols $\wedge$ and $\vee$ denote min and max respectively.
The symbol $\log$ denotes the logarithm base $\e$.
We define the binary entropy function $H$ by
\begin{equation*}
H(p) \defeq p \log \frac 1p + (1-p) \log \frac{1}{1-p}\,.
\end{equation*}

\subsection{Organization of the remainder of the paper}
In \cref{sec:smm}, we formalize the version of the second-moment method that we will employ, and show how to derive \cref{mnp}.
\Cref{sec:stein} introduces Stein's method, and gives the proofs of our central approximation results.
The appendices contain additional technical proofs and lemmas.

\section{Second Moment Method}\label{sec:smm}
The crux of our argument is the second-moment method.
Our approach requires two pieces.
The first, standard step consists in applying the second-moment method conditionally to ensure that the second moment is not dominated by rare events.
We use the following variant of the Paley--Zygmund inequality:
\begin{lemma}[{Conditional Paley--Zygmund~\cite[Theorem 2.1]{Jan96}}]\label{pz}
Let $Z = Z(n)$ be a sequence of nonnegative, integer-valued random variables, and let $W$ be another random variable on the same probability space.
If~$\PP{\EE[Z \mid W] = 0} \to 0$ and
\begin{equation*}
\frac{\EE[Z^2 \mid W]}{\EE[Z \mid W]^2} \overset{p}{\longrightarrow} 1\,,
\end{equation*}
then $Z > 0$ with high probability.
\end{lemma}
The second step consists of accurately computing the second moment of the conditional distribution, which is the main challenge in our setting.
We give a version of the second moment method (similar to Lemma 3 of \cite{achlioptas2002asymptotic}) tailored for general random constraint satisfaction problems that highlights this aspect. Say a matrix has exchangeable columns if its distribution is invariant under permutations of the columns.

\begin{lemma}[Second Moment Method for Rectangular CSPs]\label{2mm}
Let $\M$ be an ensemble of $m \times n$ matrices with independent rows, exchangeable columns, and $m = \oo(n)$. Let $A\sim \M$ and fix sets $\K_i \subseteq \mathbb Z$ for $i \in [m]$.

Define $G_{i} := \{u \in \mathcal B: \la u, A_i \ra \in \K_i \}$, and let $Z := |\bigcap_iG_{i}|$ be the number of elements of $\mathcal B$ whose inner product with the $i$th row of $A$ lies in $\K_i$ for all $i \in [m]$. For an arbitrary pair of balanced vectors $u$ and $v$ which agree on $\beta n$ coordinates, denote
\[
    \psi_{i} = \PP[A\sim \M]{u \in G_{i}}, \quad \phi_{i}(\beta) =\PP[A\sim \M]{u \in G_{i},~ v \in G_{i}}
\]
Suppose that the following conditions hold for $n$ sufficiently large.
\begin{itemize}
    \item (First Moment:)
    There exists a positive constant $c$ such that
    \begin{equation}\label{eq:2mm-goals-expectation}
        \log\E{Z} > c n
    \end{equation}
    \item (Weak Bound:) For any $\delta \in (0, 1/2)$, there exists a positive constant $C_\delta$ such that
    \begin{equation}\label{eq:2mm-goals-weak}
        \phi_i\left(\beta\right) \leq C_\delta \psi_i^2 \quad \forall i \in [m],\, \forall \beta \in [\delta, 1 - \delta]
    \end{equation}
    \item (Strong Bound:) There exists positive universal constants $C$ and $\eps$ such that
    \begin{equation}\label{eq:2mm-goals}
        \phi_i\left(\frac{1}{2}+x\right) \le \mo{\frac{1}{m}} \left(1 + Cx^2\right) \psi_i^2 \quad \forall i \in [m]\,, \forall |x| < \eps
    \end{equation}
\end{itemize}
Then the second moment method succeeds: $\E{Z^2} = (1 + \oo(1))\E{Z}^2$.
\end{lemma}
\Cref{2mm} is proved in \cref{proofs}.

In the proofs of \cref{main,fullcurve,,mnp}, the first-moment bound~\eqref{eq:2mm-goals-expectation} and weak bound~\eqref{eq:2mm-goals-weak} will follow by simple approximations.
In the dense case (\cref{mnp}), the strong bound is straightforward as well, by applying an Edgeworth expansion for lattice random walks.
However, in the context of \cref{main,fullcurve}, when we no longer assume that $m/np \to 0$, proving the strong bound~\eqref{eq:2mm-goals} directly is difficult.
While it is easy to show that~\eqref{eq:2mm-goals} holds with a multiplicative error of $1 + \cO{\frac{1}{np}}$, upgrading the error in~\eqref{eq:2mm-goals} to $1 + \cO{\frac{1}{n}} = 1 + \oo\left(\frac 1m \right)$ is the key challenge.

Our main technical idea is to establish \eqref{eq:2mm-goals} through Stein's method.
By Bayes's rule, \eqref{eq:2mm-goals} is equivalent to bounding the difference in the probability mass assigned to the event $u \in G_{i}$ by the law of $u$ versus the law of $u$ conditioned on the event $v \in G_{i
}$.
Stein's method is a powerful tool for proving quantitative comparisons between probability distributions of exactly this type.
We survey this approach and show how to derive \cref{main,fullcurve} in \cref{sec:stein}.
In the remainder of this section, we consider the simpler dense case and use \cref{2mm} to prove \cref{mnp} via a discrete Edgeworth expansion. \\

The Edgeworth expansion we need for \cref{mnp} is for the following simple, lazy random walk:
\begin{definition}
    Let $\{X_i\}$ be i.i.d.~random variables, each supported on $\{-1,0,1\}$ with densities at those points of $p(1-p)$, $p^2 + (1-p)^2$, and $p(1-p)$ respectively. Denote the variance of $X_1$ as $\sigma^2 := 2p(1-p)$ and the distribution of the lazy random walk
    $\sum_{i=1}^r X_i$ by $R(r,p)$.
\end{definition}
Since $X_i$ can only take three values, its cumulants are easy to compute. The odd cumulants vanish and the first two even cumulants are given by $\kappa_2 = \sigma^2$ and $\kappa_4 = \sigma^{2}-3\sigma^4$. Then, as $r \to \infty$, \cref{edgeworth} directly yields the following local central limit theorem

\begin{proposition}\label{prop:edgeworth_lazywalk}
Let $V \sim R(r,p)$. For any constant $k$:
\begin{equation*}
    \PP{V = k} = \frac{e^{-k^2/(2 r\sigma^2)}}{\sqrt{2\pi r\sigma^2}}\left(1 + \frac{(k^4-6k^2 + 3)(\sigma^{-2}-3)}{24r} + \cO{\frac{1}{r^{2}\sigma^4}}\right) \\
\end{equation*}
\end{proposition}

    \subsection{Proof of Theorem \ref{mnp}}\label{sec:pf-thm-mnp}
    We first show that $\E{Z_0^2|P} = (1 + \oo(1))\E{Z_0|P}^2$ when $m = \oo(np)$.
    Let $A$ be drawn from the $(m,n,p)$-Bernoulli ensemble conditioned on the event $P$ that each row of $A$ sums to an even number. Define the sets $\K_i := \{0\}$ for all $i$, and $G_i$ as in \cref{2mm}. Note the rows of $A$ are i.i.d., so we can suppress the subscripts $i$ (e.g. in $\psi_i$, $\phi_i$, and $G_{i}$) without ambiguity.

    Consider the distribution of a single row of $A$ without conditioning on $P$.
    Let $u \in \mathcal{B}$ be a balanced vector (\cref{def:balanced}), and assume without loss of generality that the first $n/2$ coordinates of $u$ are $+1$ and the last $n/2$ are $-1$. The number of ones in the first $n/2$ coordinates of $A_1$ is distributed as $\Bin\left(\frac{n}{2},p\right)$, as is the number of ones in the last $n/2$ coordinates. Thus, adding these together, $\la u, A_i \ra$ has exactly the distribution of the lazy random walk $R(n/2,p)$.

    In this notation, $\PP{u \in G} = \PP{U = 0}$ where $U \sim R(n/2,p)$. We can now compute the probability that a generic balanced vector $u$ has $u \in G$. Since the event $\{ u \in G\}$ contains the event $P$,
    \begin{equation} \label{eq:mnp-psi-def}
        \psi = \PP{u \in G | P} = \PP{u \in G}/\PP{P} = \PP{U = 0} / \PP{P}
    \end{equation}
    Before giving an asymptotic expression for $\psi$, let us derive the corresponding expression for $\phi$.
    Let $\del \in (0, 1/2)$ be an arbitrary constant.
    Fix some integer $r \in \{0,...,n/2\}$ such that $2r/n \in (\del, 1-\del)$, and consider a pair of balanced vectors $v,w \in \{\pm 1\}^n$ that agree on $2r$ coordinates. Denote the set of indices on which they agree as $S$. Then, the following events are equal:
    \begin{align*}
        \{v \in G,~ w\in G\} &= \left\{\left(\sum_{j \in S} v_jA_{ij} + \sum_{j \in S^c} v_jA_{ij}\right) \in \K,~ \left(\sum_{j \in S} v_jA_{ij} - \sum_{j \in S^c} v_jA_{ij}\right) \in \K \right\}    \\
        &=  \left\{\sum_{j \in S} v_jA_{ij} = 0,~ \sum_{j \in S^c} v_jA_{ij}= 0\right\}    
    \end{align*}
    
    Note that before we condition on $P$, the random variables $\sum_{j \in S} v_jA_{ij}$ and $\sum_{j \in S^c} v_jA_{ij}$ are independent with respective distributions $R(r,p)$ and $R\left(\frac{n}{2}-r, p\right)$. So, define two independent random variables $V$ and $V'$ with $V \sim R(r,p)$ and $V' \sim R(n/2-r,p)$. Then, since both $\{v \in G\}$ and $\{w \in G\}$ contain the event $P$, 
    \begin{equation}\label{eq:mnp-phi-def}
     \phi\left(\frac{2r}{n}\right) = \PP{v \in G,~w\in G|P} = \PP{V = 0,~ V' = 0}/\PP{P}
    \end{equation}
    Now we give asymptotic expressions for $\psi$ and $\phi$ by evaluating $\PP{P}$ and using the LCLT given in \cref{prop:edgeworth_lazywalk}. Since we have assumed $2r/n \in (\del, 1-\del)$, we have that $r = \cT{n}$. This implies that $r \sigma^2 = \cT{np}$, and since $np = \omega(1)$ we obtain
    \begin{equation}\label{eq:mnp-edgeworth}
        \PP{V = 0} = \frac{1}{\sqrt{2\pi r\sigma^2}}\left(1 + \frac{\sigma^{-2}-3}{8r} + \co{\frac{1}{np}}\right)
    \end{equation}
    
    Next, recall $P$ is the event all rows of the matrix $A$ have even sums. Let $P =: \bigcap_{i=1}^m P_i$, where $\{P_i\}$ are the iid events that each row $i$ has an even sum. We also introduce the abbreviation $\beta \defeq 2r/n$, where $\beta \in (\del, 1 - \del)$. Returning to the definitions of $\phi$ and $\psi$ \eqref{eq:mnp-psi-def} and \eqref{eq:mnp-phi-def} and applying \eqref{eq:mnp-edgeworth} yields
    \begin{align}
        \begin{split}\label{eq:mnp-psi-phi-lclt}
        \psi^2 &= \frac{1}{\pi n \sigma^2 }\left(1 + \frac{\sigma^{-2}-3}{2n} + \co{\frac{1}{np}}\right)/\PP{P}^2 \\
        \phi(\beta) &= \frac{1}{\pi n \sigma^2 \sqrt{\beta(1-\beta)} }\left(1 + \frac{\sigma^{-2}-3}{4n\beta(1-\beta)} + \co{\frac{1}{np}}\right)/\PP{P} \\ 
        \end{split}
    \end{align}
    It remains to compute $\PP{P_1}$. We claim:
    \begin{equation}\label{eq:mnp-parity-asymp}
        \PP{P_1} = \frac{1}{2} + \co{\frac{1}{np}}
    \end{equation}
   To see this, we first note that a simple induction gives the probability that a binomial $(n,p)$ random variable is of even parity:
    \begin{equation*}
        \PP{P_1} = \frac{1}{2} + \frac{1}{2}(1-2p)^n\,.
    \end{equation*}
    Since $p \le 1/2$, certainly $\PP{P_1} \ge 1/2$, and since $np = \omega(1)$, we have
    \begin{equation*}
	(1-2p)^n \leq e^{-2np} = \co{\frac{1}{np}}\,.
	\end{equation*}    
    Now, applying \eqref{eq:mnp-parity-asymp} to \eqref{eq:mnp-psi-phi-lclt}, we obtain
     \begin{align}
        \begin{split}
        \psi^2 &= \frac{4}{\pi n \sigma^2 }\left(1 + \frac{\sigma^{-2}-3}{2n} + \oo\left(\frac{1}{np}\right)\right) \\
        \phi(\beta) &= \frac{2}{\pi n \sigma^2 \sqrt{\beta(1-\beta)} }\left(1 + \frac{\sigma^{-2}-3}{4n\beta(1-\beta)} + \oo\left(\frac{1}{np}\right)\right) 
        \end{split}
    \end{align}
In particular, if $m = \oo(np)$, we have
    \[
        \frac{\phi(\beta)}{\psi^2} = \mo{\frac{1}{m}}\frac{1}{2\sqrt{\beta(1-\beta)}}
    \]
    We now verify the conditions of \cref{2mm}.
    First, since
    \begin{equation*}
\EE[Z_0 | P] = \binom{n}{n/2} \psi^m = 2^{n\mo{1}} \left( \frac{2}{\sqrt{\pi n \sigma^2} }\right)^m \left(1 + \cO{\frac{1}{np}}\right)^m = \exp\left(n \log 2 - \frac m2 \log np + \oo(n)\right)\,,
\end{equation*}
we have that $\log \EE[Z_0 | P] > cn$ so long as $n \geq C m \log m$ for any constant $C > (2 \log 2)^{-1}$.
This shows that \eqref{eq:2mm-goals-expectation} holds.
Both \eqref{eq:2mm-goals-weak} and \eqref{eq:2mm-goals} follow from the fact that for any $\delta \in (0, 1/2)$, there exists a constant $C_\delta$ such that
\begin{equation*}
\frac{1}{2\sqrt{\beta(1-\beta)}} \leq 1 + C_\delta(\beta - 1/2)^2 \quad \forall \beta \in [\delta, 1-\delta]\,,
\end{equation*}
so that
\begin{equation*}
\phi\left(\frac 12 + x\right) \leq \mo{\frac{1}{m}}(1+ C_\delta x^2)\psi^2 \quad \forall |x| \leq \delta\,.
\end{equation*}
This proves the strong bound~\eqref{eq:2mm-goals} and, \emph{a fortiori}, the weak bound~\eqref{eq:2mm-goals-weak}.

    So, if $m= \oo(np)$ and $n \geq C m \log m$ for $C > (2 \log 2)^{-1}$, then \cref{2mm} yields $\E{Z^2|P} = (1 + \oo(1))\E{Z|P}^2$. In particular, $\disc(A) = 0$ with high probability over $A$ from the $(m,n,p)$-Bernoulli ensemble conditioned on $P$.
    
    This completes one direction of the theorem. Now assume $m/(np) = \cOm{1}$. In particular, we must have $p = \oo(1)$, so we may assume $\sigma^{-2}-3 > \sigma^{-2}/2$ for $n$ sufficiently large. Since $1/(\beta(1-\beta)) \ge 4$, we obtain from~\eqref{eq:mnp-psi-phi-lclt} that for $n$ sufficiently large and $\beta$ bounded strictly from $0$ and $1$,
    \begin{equation}\label{eq:mnp-unif-bd-ratio}
        \frac{\phi(\beta)}{\psi^2} > \frac{1}{2\sqrt{\beta(1-\beta)}}\left(1 + \frac{\sigma^{-2}-3}{2n}\right) \ge 1 + \frac{\sigma^{-2}}{4n}
    \end{equation}
    We can compute the second moment the same way as in Lemma \ref{2mm}:
    \begin{align*}
        \frac{\E{Z_0^2|P}}{\E{Z_0|P}^2} &= \binom{n}{n/2}^{-1}\sum_{r = 0}^{n/2} \binom{n/2}{r}^2 \left(\frac{\phi\left(\frac{2r}{n}\right)}{\psi^2}\originalright)^m \\
        &\geq \binom{n}{n/2}^{-1}\sum_{\del n/2 < r < (1-\del)n/2} \binom{n/2}{r}^2 \left(\frac{\phi\left(\frac{2r}{n}\right)}{\psi^2}\originalright)^m
    \end{align*}
    We can uniformly lower bound $\phi(2r/n)/\psi^2$ in this last sum with \eqref{eq:mnp-unif-bd-ratio}. This will leave a sum of squared binomial coefficients. By standard tail bounds on the binomial coefficient (\cref{cramer}),
    \[
        \sum_{\del n/2 < r < (1-\del)n/2} \binom{n/2}{r}^2 = \mo{1} \sum_{r=0}^{n/2}\binom{n/2}{r}^2 = \mo{1}\binom{n}{n/2}
    \]
    Thus, the second moment is exponentially too large: since $\sigma^2 = \cT{p}$,
    \begin{align*}
        \frac{\E{Z_0^2|P}}{\E{Z_0|P}^2} \geq \mo{1}\originalleft(1 + \frac{1}{2n\sigma^2}\originalright)^m = \exp\left\{\cOm{\frac{m}{np}}\right\}\,.
    \end{align*}
\QEDA \\

\section{Stein's Method for the bounding the second moment}\label{sec:stein}
In this section, we prove our main results (\cref{main,fullcurve}) by using Stein's method to establish the inequality~\eqref{eq:2mm-goals} for the Bernoulli and Poisson ensembles in the sparse regime.

To describe our approach, we begin by rewriting~\eqref{eq:2mm-goals} as
\begin{equation}\label{eq:2mm-conditional}
\PP[A\sim \M]{u \in G_i \big| v \in G_i} \leq \PP[A\sim \M]{u \in G_i} \cdot (1 + C (\beta - 1/2)^2) \cdot \mo{m^{-1}}\,.
\end{equation}
The right side involves the law of $\langle u, A_i\rangle$, and the left side involves the law of this same random variable, conditioned on the event that $\langle v, A_i \rangle$ takes particular values, where $v$ is another balanced vector which agrees with $u$ in $\beta n$ coordinates.
Let us write $\mu_0$ and $\mu_c$ for the unconditioned and conditioned distribution, respectively, and write $\EE_0$ and $\EE_c$ for the corresponding expectation operators.
Our key ingredients are:
\begin{enumerate}
\item Two operators $T_0$ and $T_c$, satisfying
\begin{equation*}
\E[0]{T_0 f} = \E[c]{T_c f} = 0 \quad \forall f\,.
\end{equation*}
\item A function $f_{\K}$ satisfying the equation $T_0 f_{\K} = \ind_{\K} - \mu_0(\K)$.
\item A proof that
\begin{equation}\label{eq:steins_general}
\E[c]{T_0 f_\K} \leq \mu_0(\K) \cdot C (\beta - 1/2)^2 \cdot \mo{m^{-1}}\,.
\end{equation}
\end{enumerate}
The final inequality~\eqref{eq:steins_general} implies~\eqref{eq:2mm-conditional}, since $T_0 f_{\K} = \ind_{\K} - \mu_0(\K)$ and therefore
\begin{equation*}
\E[c]{T_0 f_\K} = \mu_c(\K) - \mu_0(\K)\,.
\end{equation*}
To prove~\eqref{eq:steins_general}, we will use the fact that $\E[c]{T_c f_\K} = 0$, so that $\E[c]{T_0 f_\K} = \E[c]{(T_0 - T_c) f_\K}$.
We will therefore define $T_0$ and $T_c$ in such a way that $(T_0 - T_c)$ is easy to control.

This section is organized as follows: first, we define the Stein operator and give general conditions under which we can invert a functional equation of the form $T f_{\K} = \ind_{\K} - \mu_0(\K)$. Next, we prove \eqref{eq:2mm-goals} holds for the degree-conditioned Poisson ensemble, completing Theorem \ref{fullcurve}. Finally, using the same techniques, we will prove \eqref{eq:2mm-goals} holds for the Binomial ensemble for zero-discrepancy solutions, yielding Theorem \ref{main}. 

\subsection{Stein Operator}
Let $(S,S')$ be some exchangeable random variables taking values in $[w] := \{0,1,...,w\}$, with common distribution $\mu_0$.

\begin{definition}[Stein Operator]
    Fix a constant $z$. Define the anti-symmetric operator $\antisym$ and the corresponding Stein Operator $T_0$ by
    \begin{equation}\label{eq:def-stein-op}
        (\antisym f)(S, S') := z\left(f(S') \ind_{S' > S} - f(S) \ind_{S > S'}\right), \quad T_0f(S) := \E[0]{\antisym f (S, S') | S}
    \end{equation}
\end{definition}

Note in particular that $\E[0]{T_0f(S)} = 0$ for any bounded $f$. Now, we would like that $T_0$ is invertible for the particular class of $f$ needed to examine indicator functions.
The following lemma collects the facts we will use about the inverse of $T_0$.

\begin{lemma}[Stein Operator Inverse]\label{stein-inv}
Let $\{a_i\}_{i=0}^w$ and $\{b_i\}_{i=0}^w$ be some sequences that are strictly decreasing and increasing respectively with $a_w = b_0 = 0$ and all other values strictly positive.
Define the probability distribution $\mu$ on $[w]$ and the operator $T$ on functions from $[w]$ to $\RR$ by
\begin{equation}\label{eq:lemma-stein-inv-condit}
    \mu(\{s\}) = \mu(\{0\}) \prod_{i=1}^{s}\frac{a_{i-1}}{b_i}, \quad T f(s) = a_s f(s+1) - b_s f(s) \quad \forall s \in [w]\,,
\end{equation}
where $\mu(\{0\})$ is uniquely determined by the requirement that $\mu$ have total mass $1$.

Let $\Delta$ be the difference operator given by $\Delta f(s) := f(s+1) - f(s)$ for all $s \in [w]$. For any $t \in \{1,2,...,w-1\}$, there exists a bounded function $f:[w]\to \RR$ with the following properties:
\begin{enumerate}
    \item[a. ] (Existence of an inverse) $T_\mu f(s) = \ind(s = t) - \mu(\{t\})$
    \item[b. ] (Monotonicity) $f$ is non-increasing everywhere except between $t$ and $t+1$, where it is increasing. Furthermore, $f(s)$ is non-positive when $s \le t$ and non-negative when $s \ge t+1$. 
    \item[c. ] (Uniform control)
    \[
        \sup_{0 \le s \le w-1} \left|\Delta f(s)\right| = \Delta f(t) \le \min\left(a_{t}^{-1}, b_{t}^{-1}\right)
    \]
    \item[d. ] ($L^1$ bound on $\Delta f$) 
    \begin{equation}\label{eq:Df-gaussian-sum}
        \sum_{s=0}^{w}|\Delta f(s)| = \cO{|\Delta f(t) |}
    \end{equation}
\end{enumerate}
\end{lemma}
All except the last claim, a trivial corollary of the other three, appear as Lemma 1.1.1 and Lemma 9.2.1 of the monograph of Barbour et al.~\cite{barbour1992poisson}.
For completeness, we include a full proof in \cref{proofs}.

We are now ready to prove our two main theorems.

\subsection{Proof of \cref{fullcurve}}
We will apply \cref{2mm} to the Poisson ensemble conditioned on having fixed row-sums. Under this conditioning, the law of $A$ will still have independent rows and exchangeable columns, meeting the requirements of \cref{2mm}.
We will prove the following theorem, which, when combined with  \cref{pz} and \cref{parity_coupling}, will yield \cref{fullcurve}.

\begin{theorem}\label{fullcurve_smm}
    Fix some even, non-negative numbers $(w_i)_{i=1}^m \in (2\mathbb{N})^m$ as well as some $\{r_i\}_{i=1}^m \in \mathbb N^m$ with $r_i = \cO{\sqrt{w_i}}$ for all $i$. Let $A$ be drawn from the $(m,n,\lambda)$-Poisson ensemble with $m = \oo(n)$, and denote by $W = (W_1, \dots, W_m)$ the vector of row-sums:
    \begin{equation*}
        W_i = \sum_{j=1}^n A_{ij} \quad \forall i \in [m]\,.
    \end{equation*}
    For each $i$, define $\K_i = \{0, \pm 1, ..., \pm r_i\}$, and define $G_i$ and $Z$ as in \cref{2mm}. Then, the law of $A$ conditioned on $\{W =w\}$ satisfies the strong and weak bounds of \cref{2mm}. In particular, if the $r_i$ are such that $\log \E{Z|W=w} > cn$ for a constant $c > 0$, then
    \[
    \E{Z^2|W=w} = \mo{1}\E{Z|W=w}^2\,.
    \]
\end{theorem}

Let us begin by proving that \cref{fullcurve_smm} implies \cref{fullcurve}. The proof of this technical implication can be skipped on first reading.

\begin{proof}[Proof of \cref{fullcurve} from \cref{fullcurve_smm}]
By \cref{parity_coupling} and the Paley-Zygmund inequality, it suffices to show $\E{Z_r^2|P} = \mo{1}\E{Z_r|P}^2$. Our plan is to apply \cref{fullcurve_smm} to establish the assumptions of \cref{2mm}. \\

We begin by approximating $\E{Z_r|P}$ with a similar argument used to reach \eqref{eq:fmm_heuristic}. We may assume that $r$ is an even integer. Fix $r \in 2\NN$ and let $u \in \mathcal{B}$ be such that the first half of $u$ is $+1$ and the second half is $-1$. If $\lambda n = \omega(1)$, we will apply the classical CLT; if $\lambda n = \cO{1}$, we will use the fact that any row is empty with constant probability. Either way, for some universal constant $a_0>0$, some function $\delta$ with $\lim_{x \to \infty} \delta(x) = 0$, and sufficiently large $n$: 
\begin{align}
    \E{Z_r|P} = \binom{n}{n/2}\PP{\|Au\|_\infty \le r~|~P} &= \binom{n}{n/2}\sum_{t=-r/2}^{r/2} \PP{\sum_{j=1}^{n/2} A_i - \sum_{j=n/2+1}^n A_i = 2t} \nonumber \\
    &= 2^{n + \oo(n)} \originalleft(\left(1 + \del(\lambda n)\right)\frac{a_0(r+1)}{\sqrt{\lambda n}} \wedge 1\originalright)^m \nonumber \\
    &=2^{n + \oo(n)} \originalleft(\frac{r+1}{\sqrt{\lambda n}} \wedge 1\originalright)^m \label{eq:po-fmm}
\end{align}
The last equality follows by noting $c^m = 2^{\oo(n)}$ for any fixed $c>0$. By assumption, $\E{Z_r|P} > e^{c_0n}$ for some $c_0>0$ for all $n$ sufficiently large. Because $m = \oo(n)$, \eqref{eq:po-fmm} implies that for any constant $a > 0$,   $\E{Z_{a\sqrt{\lambda n}}|P} > e^{c_1n}$ for a constant $c_1 > 0$ depending only on $a$. So, for instance, if $r > \e^{-1} \sqrt{ \lambda n}$ and the assumptions of \cref{fullcurve} are met, redefining $r := \e^{-1} \sqrt{\lambda n}$ still satisfies the assumptions of \cref{fullcurve} and provides a better upper-bound on discrepancy. Thus, assume without loss of generality $r \le \e^{-1} \sqrt{\lambda n}$. 
Now consider the three possible cases: \\

\paragraph{\textbf{Case 1:} $\lambda n = \omega\left(\log n\right)$.} Define the set $\WW \defeq \{w \in (2 \mathbb N)^m:  \forall i\,\, |w_i - \lambda n| \le \sqrt{\lambda n \log n}\}$. By \cref{poisson-tail}, $\PP{W \in \WW \mid P} = 1 - \oo(1)$. And, if $w \in \WW$, then $w_i = \lambda n \mo{1}$ for all $i$. Set $r_i := r$ for all $i$; since $r \le \e^{-1} \sqrt{\lambda n} = \e^{-1}\mo{1}\sqrt{w_i}$, we have $r_i \le 2 \e^{-1} \sqrt{w_i}$ for all $i$. Similar to \eqref{eq:po-fmm}, we have by \cref{demoivre},
\begin{align}\label{eq:po-fmm-condit}
    \E{Z_r|W=w,~P} &= \binom{n}{n/2} \prod_{i=1}^{m} \sum_{t = -r/2}^{r/2} \binom{w_i}{w_i/2+t} 2^{-w_i} = e^{n\log 2 + \oo(n)}\prod_{i=1}^m  \left(\frac{r+1}{\sqrt{w_i}} \wedge 1\right)
\end{align}
Since $w_i = \mo{1}\lambda n$, $\E{Z_r|W=w,P} = e^{-\cO{m}}\E{Z_r|P} > e^{(c-\oo(1))n}$. Thus, for all $w \in \WW$ and all $n$ sufficiently large, $\E{Z|W=w,P} > e^{c_2n}$ for some universal constant $c_2 > 0$ that in particular does not depend on $w$. So, we may apply \cref{fullcurve_smm} with $r_i = r$ for all $i$ to conclude that $\E{Z^2|W = w} = \mo{1}\E{Z|W = w}^2$ uniformly over $w \in \WW$.
Since, conditioned on $P$, $W \in \WW$ with high probability, $\E{Z^2|W,P} /\E{Z|W,P}^2 \to 1$ in probability. Applying \cref{pz}, we are done.\\

\paragraph{\textbf{Case 2:} $m \log (\log (n)) = \oo(n)$ and $\lambda n = \cO{\log n}$} Since we are upper-bounding the discrepancy of $A$, assume $r = 0$ since no better bound is possible. Trivially $r_i = \cO{\sqrt{w_i}}$, so \eqref{eq:po-fmm-condit} is available. Define $\WW := \{w: \forall i\,, w_i \leq \log(n)^2\}$. By \cref{poisson-tail}, $\PP{W \in \WW \mid P} = 1 - \oo(1)$. By \eqref{eq:po-fmm-condit}, for all $w \in \WW$ and $n$ sufficiently large,
\[
    \E{Z_0|W = w,P} \ge e^{n\log 2 + \oo(n)} \prod_{i=1}^m \left( \frac{1}{\log n}\right) > e^{n \log 2 - a_1m\log (\log (n))}
\]
Here $a_1$ is another positive constant. By assumption, $n = \omega(m\log \log n)$, so in particular obtain $\E{Z_0|W = w,P} > e^{.9n\log 2}$ for $n$ large enough. Setting $r_i = 0$ for all $i$ and invoking \cref{fullcurve_smm} for each $w \in \WW$ yields $\E{Z^2|W = w} = \mo{1}\E{Z|W = w}^2$ uniformly over $w \in \WW$.
As above, an application of \cref{pz} yields the claim.\\

\paragraph{\textbf{Case 3:} $m \log(\log( n)) = \cOm{n}$ and $\lambda n = \cO{\log n}$} We would like to again restrict to $W \in \WW$ for some asymptotically full measure set $\WW$, and then apply \cref{fullcurve_smm}. However, when $\lambda n$ is small, the row-sums of $A$ do not concentrate well. If we naively set $r_i = r$ for all $i$, we cannot hope to satisfy the theorem's condition that $r_i =\cO{\sqrt{w_i}}$ for all $i$. In order to apply \cref{fullcurve_smm}, we focus on a specific (random, $W$-measurable) subset $Z_r'$ of $Z_r$. Recall, in the notation of \cref{2mm}:
\[
    Z_r := \left|\bigcap G_i\right|, \quad  G_i:= \{u \in \mathcal{B}:~ \langle u, A_i \rangle \in \K\}, \quad \K := \ZZ \cap [-r,r]
\]
Define the ($W$-measurable) random variables $r_i' := r \wedge  \sqrt{W_i}$ for each $i$. Then, construct the corresponding random sets:
\[
    Z_r' := \bigcap G_i', \quad  G_i':= \{u \in \mathcal{B}:~ |(Au)_i| \in \K_i'\}, \quad \K_i' := \ZZ \cap [-r_i',r_i']
\]
Certainly $Z_r' \subset Z_r$, so if $Z_r' \neq \emptyset$ then $\disc(A) \le r$. Since $r_i' \le \sqrt{W_i}$ for all $i$, applying \cref{fullcurve_smm} to the set $Z_r'$ and the law of $A$ conditioned on $\{W = w\} \cap P$ establishes the weak and strong bounds of \cref{2mm}. Only the first-moment condition remains: if we can show $\log\E{Z_r'|W = w} = \cT{n}$ for all $w$ in some $\WW$ with $\PP{W \in \WW|P} = \mo{1}$, then \cref{2mm} may be applied and $Z_r' \neq 0$ with high probability by \cref{pz}. \\

We now turn our attention to establishing the first-moment condition. Observe $W_i = 0$ with probability $\cOm{e^{-\lambda n}}$; this provides a trivial lower bound on the probability that $u \in \mathcal{B}$ satisfies $|(Au)_i| \le r$. If $\lambda m = \oo(1)$, 
\[
    \E{Z_r'|P} \ge 2^{n-\oo(n)} e^{-\cO{\lambda n m}} = 2^{n - \oo(n)} 
\]
This would finish the proof, so assume that $\lambda m = \cOm{1}$, and thus $\lambda n = \omega(1)$. By the superposition property for independent Poisson variables, $W_i \sim \Pois(n\lambda)$ for each $i$, and $\sum W_i \sim \Pois(nm\lambda)$. For a sufficiently large constant $C >0$, set
\[
    \WW:= \left\{w \in (2\NN)^m : \sum_{i=1}^m w_i \le \lambda m n  +C \sqrt{\lambda m n \log n}\right\} 
\]
By \cref{poisson-tail}, we have $\PP{W \in \WW|P} = \mo{1}$. Note $\sum w_i / m = n\lambda \mo{1}$ for $w \in \WW$ because $\log (n)/ m = \oo(1)$. We now bound $\E{Z_r'|W=w}$ uniformly over $w \in \WW$. Observe that if $w_i = 0$, then $|(Au)_i| = 0 \le r_i'$ with probability $1$. Returning to \eqref{eq:po-fmm-condit} and recalling the definition of $r_i'$, 
\begin{align}
    \min_{w \in \WW}\E{Z_r'|W=w,P} &= 2^{n + \oo(n)} \min_{w \in \WW} \prod_{\substack{i:~1 \le i \le m \\ w_i \neq 0}} \left(\frac{(r \wedge \sqrt{w_i})+1}{\sqrt{w_i}}\right) \nonumber \\
    &= 2^{n + \oo(n)} \min_{w \in \WW} \prod_{\substack{i:~1 \le i \le m \\ w_i \neq 0}} \left(\frac{r}{\sqrt{w_i}} \wedge 1\right) \label{eq:po-min-fmm}
\end{align}
We claim that if $r < \sqrt{w_j}$, then we may assume $w_j = 0$. Indeed, consider $w \in \WW$ with $r > \sqrt{w_j} > 0$ for some $j$. Define $w'$ as a copy of $w$ with a modification: $w_j' = 0$ and for arbitrary $k \neq j$, let $w_k' = w_k + w_j$. Then, a simple computation yields
\begin{align*}
    \prod_{\substack{i:~1 \le i \le m \\ w_i \neq 0}} \left( \frac{r}{\sqrt{w_i}} \wedge 1\right) \ge
    \left( \frac{r}{\sqrt{w_k'}} \wedge 1\right) \prod_{\substack{i:~1 \le i \le m \\ w_i \neq 0,~ i \neq j, i\neq k}} \left( \frac{r}{\sqrt{w_i}} \wedge 1\right) 
    = \prod_{\substack{i:~1 \le i \le m \\ w_i' \neq 0}} \left( \frac{r}{\sqrt{w_i'}} \wedge 1\right)
\end{align*}
Repeating this argument for each index $j$ with $r > \sqrt{w_j} > 0$, we see $\E{Z_r'|W=w,P}$ is minimized (up to a $2^{\oo(n)}$ factor), by $w$ with $w_j = 0$ or $\sqrt{w_j} > r$, for all $j$. Now, fix some positive integers $S$ and $t$. By arithmetic-geometric mean inequality, 
\begin{align*}
    \max_{\{w_i\}_{i=1}^t:~ \sum w_i = S} \prod_{i=1}^t \frac{r}{\sqrt{w_i}} &\ge \prod_{i=1}^t \frac{r}{\sqrt{S/t}} = \left(\frac{tr^2}{S}\originalright)^{t/2} =: f_S(t)
\end{align*}
We apply this to \eqref{eq:po-min-fmm} for each $w\in \WW$ by letting $t$ denote the number of non-zero indices of $w$ and $S = \sum w_i$. Using our observation about the structure of $w_j$ in the second inequality,
\begin{align*}
    \min_{w \in \WW}\E{Z_r'|W = w} \ge 2^{n + \oo(n)} \min_{w \in \WW} \prod_{\substack{1 \le i \le m \\ w_i \neq 0}} \left(\frac{r}{\sqrt{w_i}}\wedge 1\right) &\ge 2^{n + \oo(n)} \min_{w \in \WW} \prod_{\substack{1 \le i \le m \\ w_i \neq 0}} \frac{r}{\sqrt{w_i}} \ge 2^{n + \oo(n)} \min_{\substack{w \in \WW \\ S := \sum w_i}} \min_{t \in [0,m]} f_S(t)
\end{align*}
It is a simple calculus exercise that $f_S(t)$ is decreasing for $0 < t < S/(r^2e)$; increasing for $t > S/r^2e$; and has a global minimum at $S/(r^2e)$. Recall $S  = \mo{1}\lambda m n$ and $r \le \e^{-1} \sqrt{\lambda n}$. Then, $S/(r^2e) > m$, so the minimum of $f_S(t)$ on $[0,m]$ is at $t = m$. Recalling the expression for $\E{Z_r|P}$ given in \eqref{eq:po-fmm},
\begin{align*}
    \min_{w \in \WW}\E{Z_r'|W = w,~P} &\ge 2^{n + \oo(n)} f_{\lambda m n }(m) = 2^{n + \oo(n)} \left(\frac{r}{\sqrt{n\lambda}}\originalright)^m = \E{Z_r|P}e^{\oo(n)}
\end{align*}
 Since $m = \oo(n)$ and $\E{Z_r|P} = e^{\cT{n}}$, then $\E{Z_r'|W  = w} = e^{\cT{n}}$ uniformly over $w \in \WW$ and we are done.
\end{proof}

\begin{proof}[Proof of \cref{fullcurve_smm}]
Our goal is to verify that the weak and strong bounds \eqref{eq:2mm-goals-weak} and \eqref{eq:2mm-goals} hold. Fix a particular $i$; we can henceforth suppress all subscripts $i$ without ambiguity. Then, $\K = \{0,\pm 1, ..., \pm r\}$ for some $r = \cO{\sqrt{w}}$. Define the set $\K_w := \{(w+k)/2,~ k\in \K\}$; while this is a slight overload of notation, it will never be ambiguous because the subscript $i$ is suppressed for the remainder of this theorem. Finally, define the shorthand $\gamma := 1 -\beta$. For the Poisson ensemble conditioned on $(w_i)$ described in the theorem statement, and for $\phi$ and $\psi$ defined in \cref{2mm},
\begin{align*}
    \psi &= 2^{-w}\sum_{k \in \K}\binom{w}{(w+k)/2}, \\
    \phi(\beta) &= 2^{-w}\sum_{k,k' \in \K}\binom{w}{\frac{w+k}{2}} \sum_{c} \binom{\frac{w+k}{2}}{\frac{w+k'}{2}-c} \binom{\frac{w-k}{2}}{c}\beta^{2c + \frac{k-k'}{2}}\gamma^{w + \frac{k'-k}{2} -2c}
\end{align*}
We first verify the weak bound  \eqref{eq:2mm-goals-weak}.
Since we restrict to $\beta$ bounded strictly away from $0$ and $1$ by constants, the weak bound will follow easily from standard approximations for binomial coefficients (\cref{stirling,demoivre}).
Since $r =\cO{\sqrt w}$ and $|k| \leq r$ for each $k \in \K$, Stirling's approximation yields
\begin{equation*}
\psi = 2^{-w}\sum_{k \in \K}\binom{w}{(w+k)/2} = \sum_{k \in \K}\cT{\frac{1}{\sqrt w}} = \cT{|\K|w^{-1/2}}\,.
\end{equation*}
By a similar argument, making the substitution $c = \beta w /2 + j$ and using again the fact that $k, k' = \cO{\sqrt w}$,
\begin{align*}
    \phi(\beta) &:= \cO{\frac{1}{\sqrt{w}}}\sum_{k,k' \in \K} \sum_{c} \left(\binom{\frac{w+k}{2}}{\frac{w+k'}{2}-c} \beta^{\frac{k-k'}{2}+c}\gamma^{\frac{w + k'}{2} -c}\right)\left(\binom{\frac{w-k}{2}}{c}\beta^{c}\gamma^{\frac{w-k}{2}-c}\right) \\
    &\le \cO{\frac{1}{\sqrt{w}}}\sum_{k,k' \in \K} \sum_{|j| \le \cO{\sqrt{w\log w}} } e^{-\cT{j^2/(w\beta \gamma)}}\frac{1}{w\beta\gamma}\\
    &= \cO{\frac{|\K|^2}{w\beta \gamma}}
\end{align*}
Thus, if $\beta \in [\delta, 1-\delta]$, then $\phi(\beta) \le C_\delta \psi^2$. This completes the weak bound.

The strong bound requires a much finer quantitative estimate of $\phi$ and $\psi$ when $\beta$ is very close to $1/2$. Standard approximation techniques give a $\mO{w^{-1}}$ multiplicative error. Unless we restrict ourselves to the dense case by making the assumption that $m = \oo(w)$, this error is far too large. Instead, we will use Stein's method of exchangeable pairs to compute $\phi_i(\beta)$ in terms of $\psi_i$ for $\beta$ close to $1/2$.

Let us consider two balanced vectors, $u$ and $v$.
Since the entries of row $A_i$ of $A$ are i.i.d.~Poisson random variables, if we condition on $\{W= w\}$, then $A_i$ can be constructed by starting with the all-zeroes vector of length $n$, and then choosing $w$ coordinates uniformly at random ({\em with} replacement) to increment. By keeping track of whether $u$ is positive or negative in each chosen coordinate, we see $\la A_i, u \ra$ is characterized by a binomial random variable. Similarly, the pair $(\langle u, A_i\rangle,~\langle v, A_i \rangle)$ is characterized by a multinomial random variable counting how many outcomes in the construction of $A_i$ correspond to coordinates where $u$ and $v$ are both positive, both negative, or of mixed sign.

This description suggests the following construction.
We draw $w$ independent random variables from a categorical distribution with four outcomes, labeled $(+,+)$, $(+,-)$, $(-,-)$, and $(-,+)$, where we assign $\beta/2$ probability to each of the outcomes $(+,+)$ and $(-,-)$ and $\gamma/2$ probability to each of the outcomes $(+,-)$ and $(-,+)$.
We view these four outcomes as reflecting the signs of the entry of $v$ and the entry of $u$ corresponding to each selected coordinate.

Let $\sigma = (\Sa,\Sb,\Sc,\Sd)$ be the respective counts of how many outcomes of each type are observed. Then $\sigma$ has a multinomial distribution. By construction,
\[
    \la u, A_i \ra \stackrel{d}{=} \Sa + \Sd - \Sb - \Sc = w-2(\Sb+\Sc)\,.
\]
To obtain an exchangeable pair, we construct another tuple $\sigma'$ by selecting one of the $w$ outcomes uniformly at random and resampling it from the original categorical distribution. We call the joint law of $(\sigma, \sigma')$ generated by this procedure the \emph{unconditioned distribution}, which we denote by $\pp_0$.

Next, we consider a different process for generating $\sigma$, which reflects the law of $\la u, A_i \ra$ when we condition on the value of $\la v, A_i \ra$. Concretely, we group the four outcomes of our categorical random variables into two types depending on their first coordinate: $(+,+)$ and $(+,-)$ are one type and $(-,-)$ and $(-,+)$ are another type. Define the events
\[
    E_k := \left\{\Sa+\Sb = \frac{w+k}{2}\right\}, \quad E_\K := \bigcup_{k \in \K} E_k
\]
We draw $w$ independent random variables from the same categorical distribution as above, but we condition on the event $E_\K$ that the number of $(+, +)$ and $(+, -)$ outcomes is $(w+k)/2$ for some $k \in \K$.
This yields a new distribution on the tuple $\sigma = (\Ta, \Tb, \Tc, \Td)$ of counts.
To obtain an exchangeable pair, we can generate another tuple $\sigma'$ from $\sigma$ by picking one of the $w$ outcomes uniformly at random and resampling it from the categorical distribution \emph{conditioned on the outcome being of the same type}.
We call the joint law of the resulting pair the \emph{conditioned distribution}, which we denote by~$\pp_c$.

As above, if we view the outcomes as the signs of the entries of $v$ and $u$ corresponding to each selected coordinate, then under $\pp_c$
 \[
   \Ta + \Td - \Tb - \Tc = w-2(\Tb+\Tc)  \stackrel{d}{=} \left\{\la u , A_i \ra ~\big|~ \la v, A_i \ra \in \K \right\}
 \]

We focus on the quantity $S := S(\sigma) = \Sb + \Sc$ under the conditioned and unconditioned distributions.
We have by construction:
\begin{align}
\begin{split}\label{eq:po-stein-implies-quadUB}
    &\pp_0\left[S\in \K_w\right] = \psi = \pp_0[E_\K], \quad \pp_c\left[S \in \K_w \right]\psi = \phi(\beta) 
\end{split}
\end{align}

Thus, we want to show that the probability that $S \in \K_w$ is close under $\pp_0$ and $\pp_c$.
We will use Stein's method to compare these probabilities.
As in~\eqref{eq:def-stein-op},  define the Stein operators
\begin{equation*}
T_0 f(\sigma) = \E[0]{\antisym f(S(\sigma), S(\sigma'))|\sigma}\,, \quad T_c f(\sigma) = \E[c]{\antisym f(S(\sigma), S(\sigma'))|\sigma}\,.
\end{equation*}
Writing $\mu_0$ and $\mu_c$ for the probability measures on $\mathbb N$ induced by $S$ under $\pp_0$ and $\pp_c$,
we find a function $f$ for which $T_0 f(\sigma) = \ind_{S(\sigma) \in \K_w} - \mu_0(\K_w)$, and then we compute
\begin{equation*}
\E[c]{(T_c - T_0)f(S)} = - \E[c]{T_0 f(S)} = \mu_0(\K_w) - \mu_c(\K_w)\,.
\end{equation*}
Carrying out these constructions by means of \cref{stein-inv}, we obtain the following result whose proof is deferred to the next section.
\begin{lemma}\label{po-f-bound}
Define $\Delta$ as the one-step difference operator, $\Delta f(s) := f(s+1)-f(s)$. There exists a function $f$ satisfying
\begin{equation}\label{eq:po-final-bd}
    \mu_0(\K_w) - \mu_c(\K_w) = \E[c]{(\gamma-\beta)\left(\frac{\Sc-\Sb}{2}\right)\Delta f(S)}
\end{equation}
with the property:
\begin{equation}\label{eq:po-stein-inv-L1}
    \sum_{s = 0}^w |\Delta f(s)| = \cO{|\K|w^{-1}}
\end{equation}
\end{lemma}

It remains to bound \eqref{eq:po-final-bd}.
Let $x = \beta - 1/2$.
We establish the following proposition.
\begin{proposition}\label{po-with-indicator-bd} 
Uniformly over $s$,
\begin{equation*}
|\E[c]{(\Sc-\Sb)\ind_{S = s}}| = \cO{xw^{1/2}}\cEws\,.
\end{equation*}
\end{proposition}
The proof of \cref{po-with-indicator-bd} is the most involved part of the theorem.
This task, though technical, is significantly simplified by the fact that it suffices to estimate the quantity in question to \emph{constant} multiplicative error, whereas our original goal required estimating $\phi_i(\beta)$ to error $1 + \co{\frac 1m}$.
Before proving \cref{po-with-indicator-bd}, we first show how it implies the strong bound. Applying \cref{po-with-indicator-bd} to~\eqref{eq:po-final-bd} yields
\begin{align*}
    |\mu_0(\K_w) - \mu_c(\K_w)| &= x \left| \E[c]{(\Sc-\Sb)\Delta f(S)}\right| \\
    & = x \left|\sum_{s = 0}^w \E[c]{(\Sc-\Sb)\ind_{S = s}} \Delta f(s)\right| \\
    & \leq \cO{x^2 w^{1/2}} \sum_{s=0}^w \cEws |\Delta f(s)| \\
    & \leq \cO{x^2 w^{1/2}} \sum_{s = 0}^w |\Delta f(s)| \\
    & = \cO{x^2 |\K| w^{-1/2}}\,,
    \end{align*}
where the last step follows from~\eqref{eq:po-stein-inv-L1}.

Since $S$ under $\pp_0$ has distribution $\Bin(w, 1/2)$, the de Moivre--Laplace theorem (\cref{demoivre}) implies that if $k = \cO{\sqrt{w}}$, then $\mu_0(\{(w+k)/2\}) = \cT{\mu_0(\{w/2\})}$. In particular, this holds for all $k \in \K$. Thus, $\psi = \mu_0(\K_w) = \cT{|\K|w^{-1/2}}$. Rearranging yields $|\K| = \cT{ \mu_0(\K_w)\sqrt{w}}$. So,
\[
    |\mu_c(\K_w) - \mu_0(\K_w)| = \cO{x^2 \mu_0(\K_w)} = \cO{x^2 \psi}
\]
The last equality follows from the identity \eqref{eq:po-stein-implies-quadUB}. Using \eqref{eq:po-stein-implies-quadUB} again yields the strong bound \eqref{eq:2mm-goals} for some universal positive constant $C$.
\[
    \phi\left(\frac{1}{2}+x\right) = \psi \mu_c(\K_w) = (1+C x^2) \psi^2
\]
We have now shown all the assumptions of Lemma \ref{2mm} are satisfied. Thus the second moment method succeeds for the Poisson ensemble conditioned on $\{W = w\}$, proving \cref{fullcurve_smm}. 
\end{proof}

It remains to prove \cref{po-with-indicator-bd}, modulo some technical lemmas which we defer to the following section.

\begin{proof}[Proof of \cref{po-with-indicator-bd}]
Let $G(k,c) := \pp_c[\Sc = c, S = s | E_k] = \pp_c[\Sc = c, \Sb = s - c | E_k]$. 
Conditioned on $E_k$, $\Sc$ and $\Sb$ are independent, with $\Sb$ having distribution $\Bin((w+k)/2, \gamma)$ and $\Sc$ having distribution $\Bin((w-k)/2, \beta)$.
We can therefore write $G(k, c)$ as a product of binomial densities:
\[
    G(k,c) = \binom{(w+k)/2}{s-c}\binom{(w-k)/2 }{c} \beta^{\frac{w+k}{2} + 2c - s} \gamma^{\frac{w-k}{2} - 2c + s}
\]
Fix $s$ with $0 \le s \le w$. Our goal is to prove uniformly for $k \in \K$:
\begin{equation}\label{eq:po-E(c-b)-ind_t}
    \E[c]{(\Sc-\Sb)\ind_{S = s}| E_k} = \sum_{c = 0}^{s} (2c-s)G(k,c) = \cO{xw^{1/2}}\cEws\,,
\end{equation}
The first equality is by definition; the second is the claim that directly yields the proposition after averaging over $k \in \K$. We begin with a crude approximation for $G$. By standard binomial inequalities:
\begin{lemma}\label{po-G-tail}
    Uniformly over $c$,
    \begin{equation*}
        G(k,c) 
        =\cO{ \frac{1}{w}} \cE{-\cT{\frac{(s - w/2)^2 + (c - \beta s)^2}{w}}} 
    \end{equation*}
\end{lemma}

We will employ the following basic fact about Gaussian sums, which follows immediately upon comparison with a Gaussian integral.
\begin{lemma}\label{gaussian-sum}
Let $q \geq 0$, $r \in \RR$, and $w \geq 1$. For some implicit constant depending only on $q$,
\begin{align*}
     \sum_{y \in \ZZ} |y-r|^q e^{-(y-r)^2/w} &= \cO{w^{(q+1)/2}}, 
\end{align*}
\end{lemma}
We consider three cases separately, depending on the size of $x$. \\

\paragraph{\textbf{Case 1}: $|x| \ge w^{-1/2}$}
Here, \Cref{po-G-tail} suffices. Indeed, also applying \cref{gaussian-sum} yields
\begin{align*}
    \left|\E[c]{(\Sc-\Sb)\ind_{S = s}| E_k}\right| &= \cO{\frac{1}{w}}\sum_{c=0}^s (2c-s )\e^{-\cT{\frac 1 w}((c - \beta s)^2 + (s-w/2)^2)} \\
    &=  \cO{\frac 1 w}\Big(2\sum_{c = 0}^s(|xs| + |c - \beta s|) \e^{-\cT{\frac 1 w}(c - \beta s)^2} \Big) \cEws \\
    &=  \left[\cO{xw^{1/2}} +  \cO{1} \right] \cEws 
    \\ &= \cO{xw^{1/2}}\cEws\,,
\end{align*}
where the last step uses the assumption that $|x| \geq w^{-1/2}$.
Averaging over $k \in \K$ completes the proposition in the case $x \ge w^{-1/2}$. \\

\paragraph{\textbf{Case 2}: $|x| \leq w^{-1}$}

For this case, we need to exploit a symmetry. We have assumed that $k \in \K$ implies $-k \in \K$. And, since $S \stackrel{d}{=} w - S$, we also have $\pp_c[E_k] = \pp_c[E_{-k}]$.
Thus:
\begin{align}
    \left|\E[c]{(\Sc-\Sb)\ind_{S = s}}\right| &\leq \sum_{k \ge 0}\left|\E[c]{(\Sc-\Sb)\ind_{S = s}| E_k} + \E[c]{(\Sc-\Sb)\ind_{S = s}| E_{-k}}\right| \pp_c[E_k] \nonumber\\
    &=\sum_{k \ge 0}\pp_c[E_k] \sum_{c = 0}^s|2c-s|\left|G(k,c) - G(-k,s-c)\right| \label{eq:po-G-diff}
\end{align}

The case when $|x| \le w^{-1}$ now follows from easy arguments. We have: 
\begin{align*} 
    G(k,c) - G(-k,s-c) &= \binom{\frac{w+k}{2}}{s-c}\binom{\frac{w-k}{2}}{c}\left(\beta^{\frac{w-k}{2}+2c-s}\gamma^{\frac{w+k}{2}-2c+s} - \beta^{\frac{w+k}{2}-2c+s}\gamma^{\frac{w-k}{2}+2c-s}\right)\\ 
    &= G(k,c)\left(1 - \originalleft(\frac{\beta}{\gamma}\originalright)^{-k -4c+2s}\right) 
\end{align*}
Since $x = \cO{w^{-1}}$, note that $x(k + 4c - 2s) = \cO{1}$. Hence, observing $\beta/\gamma = 1 + \cO{x}$ yields
\begin{equation*}
\left(1 - \originalleft(\frac{\beta}{\gamma}\originalright)^{-k -4c+2s}\right) = \cO{x(k + 4c - 2s)}\,.
\end{equation*}
In total, combining this with \cref{po-G-tail}, we obtain
\begin{align*}
    \left|G(k,c) - G(-k,s-c)\right| &=\cO{x(k + 4c - 2s)}\cO{\frac{1}{w}} \cE{-\cT{\frac{1}{w}}\left[(s-w/2)^2 + (\beta s - c)^2\right]}\,.
\end{align*}
Recall $x := \beta - 1/2$. Since $x = \cO{w^{-1}}$ by assumption, $|2c-s| \leq |2c - 2 \beta s| + 1$. Applying \cref{gaussian-sum},
\begin{align*}
\sum_{c = 0}^s|2c-s|\left|G(k,c) - G(-k,s-c)\right|
    &= \cO{\frac{x}{w}}\sum_{c = 0}^{s} (|k||2c - s|  + |2c -s|^2) \e^{-\cT{\frac{1}{w}}[(s-w/2)^2 + (\beta s - c)^2]} \\
    &= \cO{\frac{x}{w}}\sum_{c =0}^s (1 + |k||c - \beta s| + |c - \beta s|^2) \e^{-\cT{\frac{1}{w}}[(s-w/2)^2 + (\beta s - c)^2]}  \\
    &= \cO{xw^{1/2}}\cEws
\end{align*}
Since this holds uniformly for $k \in \K$, returning to \eqref{eq:po-G-diff} completes the proposition for $x \le w^{-1}$. \\

\paragraph{\textbf{Case 3}: $w^{-1} \leq |x| \leq w^{-1/2}$}
For this case, we again employ a symmetrized expression.
Fix a $k \geq 0$.
As in~\eqref{eq:po-G-diff}, we have
\begin{equation*}
    \left|\E[c]{(\Sc-\Sb)\ind_{S = s}}\right| \leq \sum_{k \ge 0}\left|\E[c]{(\Sc-\Sb)\ind_{S = s}| E_k} + \E[c]{(\Sc-\Sb)\ind_{S = s}| E_{-k}}\right| \pp_c[E_k] 
\end{equation*}
First, note that we can assume that $|s - w/2| = \cO{\sqrt{w \log w}}$.
Indeed, by \cref{po-G-tail},
\begin{align*}
|\E[c]{(\Sc-\Sb)\ind_{S = s}| E_k} + \E[c]{(\Sc-\Sb)\ind_{S = s}| E_{-k}}| & \leq \sum_{c=0}^{s} |2c-s|\left|G( k, c) - G( -k, s-c)\right| \\
& \leq \cO{w}\sum_{c=0}^{s} G( k, c) + G( -k, s-c)\\
& = \cO{w^{1/2}} \, \e^{-\frac{C (s - w/2)^2}{w}}
\end{align*}
for some positive constant $C$; if $|s - w/2| \geq C' \sqrt{w \log w}$ for a sufficiently large positive constant $C'$, then
\begin{equation*}
\e^{- \frac{C (s - w/2)^2}{w}} \leq w^{-1} \e^{- \frac{C (s - w/2)^2}{2w}} = \cO{x}\e^{- \frac{C (s - w/2)^2}{2w}}\,.
\end{equation*}
Therefore, if $|s - w/2| \geq C' \sqrt{w \log w}$, we already have
\begin{equation*}
\sum_{c = 0}^s|2c-s|\left|G(k,c) - G(-k,s-c)\right| = \cO{x w^{1/2}} \cEws\,,
\end{equation*}
which is the desired bound.
We therefore assume in what follows that $|s - w/2| = \cO{\sqrt{w \log w}}$.

For this case, we develop a slightly different symmetrized expression based on~\eqref{eq:po-G-diff}. Fix some $k \ge 0$.
Writing $[xs]$ for the nearest integer to $xs$, we have
\begin{align*}
\E[c]{(\Sc-\Sb)\ind_{S = s}| E_k} + \E[c]{(\Sc-\Sb)\ind_{S = s}| E_{-k}} & = \sum_{c = 0}^s (2c - s) G(k, c) + \sum_{c = 0}^s (2c - s)  G(-k, c) \\
& = \sum_{c \in \ZZ} (2c - s) G(k, c) + \sum_{c \in \ZZ} (s + 4[xs] - 2c) G(-k, s+2[xs] - c) \\
& = \sum_{c \in \ZZ} (2c - 2\beta s) (G(k, c) - G(-k, s+2[xs] - c)) \\
& \quad + 2xs \sum_{c \in \ZZ} (G(k, c) - G(-k, s+2[xs] - c)) \\
& \quad + 4[xs] \sum_{c \in \ZZ}G(-k, s+2[xs] - c)\,.
\end{align*}

We first claim that the last two sums are small enough.
Indeed, using \cref{po-G-tail,gaussian-sum}, we see that both terms are bounded by
\begin{equation*}
\cO{\frac{xs}{w}} \sum_{c \in \ZZ} \e^{- \cT{\frac 1w}[(s - w/2)^2 + (c - \beta s)^2]} = \cO{xw^{1/2}}\cEws\,,
\end{equation*}
which is of the desired size.
Moreover, we further claim that
\begin{multline*}
\sum_{c \in \ZZ} (2c - 2\beta s) (G(k, c) - G(-k, s+2[xs] - c)) = \sum_{|c - \beta s| = \cO{\sqrt{w \log w}}} (2c - 2\beta s) (G(k, c) - G(-k, s+2[xs] - c)) \\ + \cO{xw^{1/2}}\cEws\,.
\end{multline*}
This truncation is valid because \cref{po-G-tail} guarantees that there exists a positive constant $C$ such that the portion of the sum outside the range $|c - \beta s| \leq C \sqrt{w \log w}$ contributes at most
\begin{equation*}
w^{-1/2} \cEws
\end{equation*}
to the sum; since we have assumed that $w^{-1} \leq |x|$, this error is also of size $\cO{xw^{1/2}}\cEws$.

Combining the above bounds, we obtain that
\begin{multline}\label{final_case_3}
|\E[c]{(\Sc-\Sb)\ind_{S = s}}| = \sum_{k \geq 0} \pp_c[E_k] \sum_{|c - \beta s| = \cO{\sqrt{w \log w}}} (2c - 2\beta s) (G(k, c) - G(-k, s+2[xs] - c)) \\ + \cO{xw^{1/2}}\cEws\,,
\end{multline}
so it suffices to obtain an accurate approximation of the first sum under the restriction that both $|s - w/2|$ and $|c - \beta s|$ are $\cO{\sqrt{w \log w}}$.
We use the following refinement of \cref{po-G-tail}.
\begin{lemma}\label{G-expansion}
    Define $\dd := s - \frac{w}{2}$ and $j \defeq c - \beta s$. Let $|j| = \cO{ \sqrt{w \log w}}$, $|\dd| = \cO{\sqrt{w\log w}}$, and $k \in \K$. For $x \le w^{-1/2}$, 
    \begin{equation}\label{eq:po-G-expansion}
        G(k, c) = (1+ \Ej(j,\dd)) \frac{1}{\pi\beta\gamma w}\cE{-\frac{\dd^2 + 2j^2}{\beta\gamma w }}
    \end{equation}
    where $\Ej(j,\dd)$ denotes a quantity satisfying
    \begin{equation*}
                \Ej(j,\dd) = \cO{\frac{1}{w^{1/2}}} + \cO{\frac{(|j| \vee |\dd|)^3}{w^2}} 
    \end{equation*}
\end{lemma}

Continuing to write $\dd$ and $j$ as in the statement of \cref{G-expansion}, we obtain
\begin{equation*}
(G(k, c) - G(-k, s+2[xs] - c)) = (1+ \Ej(j,\dd)) \frac{1}{\pi\beta\gamma w}\cE{-\frac{\dd^2 + 2j^2}{\beta\gamma w }} - (1+ \Ej(-\tilde j,\dd)) \frac{1}{\pi\beta\gamma w}\cE{-\frac{\dd^2 + 2\tilde j^2}{\beta\gamma w }}\,,
\end{equation*}
where
\begin{equation*}
\tilde j \defeq s + 2[xs]-c - \beta s = j + 2([xs]-xs)\,.
\end{equation*}
Since $|\tilde j - j| \leq 2$, we have in particular that $\Ej(- \tilde j, \dd) = \cO{\Ej(j, \dd)}$, as well as:
\begin{equation*}
\cE{-\frac{\dd^2 + 2\tilde j^2}{\beta\gamma w }} = \mO{\frac{|j|}{w}} \cE{-\frac{\dd^2 + 2j^2}{\beta\gamma w }} = (1+ \cO{\Ej(j,\dd)})\cE{-\frac{\dd^2 + 2j^2}{\beta\gamma w }}
\end{equation*} 
We obtain
\begin{equation}\label{eq:po-G-diff-final}
    \left|(G(k, c) - G(-k, s+2[xs] - c))\right| = \cO{\Ej(j,\dd)} \frac{1}{\pi\beta\gamma w}\cE{-\frac{\dd^2 + 2j^2}{\beta\gamma w }}
\end{equation}
By \cref{gaussian-sum},
\begin{align*}
\sum_{c \in \ZZ} |c - \beta s| \Ej(j,\dd) \frac{1}{\pi\beta\gamma w}\e^{-\frac{\dd^2 + 2j^2}{\beta\gamma w }} & = \cO{\frac 1 w} \e^{-\frac{\dd^2}{\beta \gamma w}} \sum_{c \in \ZZ} \left(\frac{|c - \beta s|}{w^{1/2}} + \frac{|c - \beta s|^4}{w^{2}} + \frac{|c - \beta s| |\dd|^3}{w^2}\right) \e^{-\cT{\frac{(c - \beta s)^2}{w}}} \\
& = \left(\cO{\frac{1}{w^{1/2}}} + \cO{\frac{|\dd|^3}{w^2}}\right) \e^{-\frac{\dd^2}{\beta \gamma w}} \\
& = \cO{\frac{1}{w^{1/2}}}\e^{-\frac{\dd^2}{2 \beta \gamma w}}\,,
\end{align*}
where the last step uses that $\cO{\frac{|\dd|^3}{w^2}} \e^{-\frac{\dd^2}{\beta \gamma w}} = \cO{\frac{1}{w^{1/2}}}\e^{-\frac{\dd^2}{2 \beta \gamma w}}$.

Since $w^{-1/2} = \cO{xw^{1/2}}$, combining this calculation with~\eqref{eq:po-G-diff-final} yields
\begin{equation*}
\sum_{c \in \ZZ} |c - \beta s| \left|(G(k, c) - G(-k, s+2[xs] - c))\right| = \cO{xw^{1/2}} \cEws\,,
\end{equation*}
and combining this fact with~\eqref{final_case_3} finishes the proof of \cref{po-with-indicator-bd}.
\end{proof}

\subsection{Proofs of lemmas}
\begin{proof}[Proof of lemma \ref{po-f-bound}]
We let $S = S(\sigma)$ and $S' = S(\sigma')$, so that $(S, S')$ is an exchangeable pair under both $\pp_0$ and $\pp_c$.
Define $a_S$ and $b_S$,
\begin{align}
\begin{split}\label{eq:po-alpha,beta}
    \pp_0\left[S' > S~|~S\right] =  \frac{w-S}{2w} =: \frac{a_S}{w}\,, \quad \pp_0\left[S' < S~|~S\right]
    = \frac{S}{2w} =: \frac{b_S}{w}\,.
\end{split}
\end{align}
On the other hand, conditioning on $E_k$ for some $k \in \K$:
\begin{align}
    \begin{split}\label{eq:po-transition-probabilities-k}
        \pp_c\left[S'<S ~|~ \Sb,\Sc, E_k\right] &= \frac{\Sb\beta + \Sc\gamma}{w}, \\ \pp_c\left[S'>S ~|~ \Sb,\Sc, E_k\right] &= \frac{(w/2-k/2-\Sc)\beta + (w/2+k/2-\Sb)\gamma}{w}
    \end{split}
\end{align}
Consider the space of functions on $[w]$.
We define a skew-symmetric operator $\antisym$ on such functions by
\begin{equation*}
    (\antisym f)(s, s') = w\left(f(s')\ind_{s' = s + 1} - f(s)\ind_{s = s'+1}\right)
\end{equation*}
We then define the Stein operators
\begin{align}
T_0 f(\sigma) & \defeq \EE_0[(\antisym f)(S, S') \mid \sigma] \\
T_c f(\sigma) & \defeq \EE_c[(\antisym f)(S, S') \mid \sigma]\,.
\end{align}
Explicitly, by~\eqref{eq:po-alpha,beta},
\begin{equation*}
T_0 f(\sigma) = a_S f(S+1) - b_S f(S)\,.
\end{equation*}
In particular, $T_0 f$ depends on $\sigma$ only through $S(\sigma)$, so that $T_0$ agrees with the operator $T$ on functions on $[w]$ defined by
\begin{equation}\label{eq:po-T0}
    T f(s) = a_s f(s+1) -b_s f(s) \quad \forall s \in [w]\,.
\end{equation}

Using the tower property of conditional expectation (i.e. the identity $\EE[\EE[X|Y]] = \EE[X]$) to condition over $E_k$ for each $k \in \K$ via \eqref{eq:po-transition-probabilities-k},
\begin{align*}
    T_c f(\sigma) - T_0 f(\sigma) &= (\gamma-\beta)\sum_{k \in \K}\left( \left(\frac{\Sc-\Sb}{2}\right)\Delta f(S) + \frac{k}{2}f(S+1)\right) \pp_c[{E_k| E_{\K}}] \\
    &= (\gamma-\beta)\left(\left(\frac{\Sc-\Sb}{2}\right)\Delta f(S) + f(S+1)\sum_{k \in \K, k \geq 0}  \frac{k}{2} (\pp_c[E_k| E_{\K}]-\pp_c[{E_{-k}| E_{\K}}])\right)
\end{align*}
Marginally, $\Sa + \Sb \sim \Bin(w, 1/2)$, so $\Sa + \Sb \overset{d}{=} w - (\Sa + \Sb)$ and therefore $\pp_c[{E_k| E_{\K}}]=\pp_c[{E_{-k}| E_{\K}}]$ for all $k \in \K$. So, the summation in the last equation cancels to zero and we obtain
\begin{equation} \label{eq:po-T-diff}
    T_c f(\sigma) - T_0 f(\sigma)= (\gamma-\beta)\left(\frac{\Sc-\Sb}{2}\right)\Delta f(S) 
\end{equation}
Since $S(\sigma)$ under $\mu_0$ is an unbiased Binomial random variable with $w$ trials, it is easy to check that taking $\mu := \mu_0(S)$, $a_s$ and $b_s$ as given in \eqref{eq:po-alpha,beta}, and $T$ as in $\eqref{eq:po-T0}$, satisfies condition \eqref{eq:lemma-stein-inv-condit}. Thus, for each $k \in \K$, we may apply Lemma \ref{stein-inv} with $t := (w+k)/2$ to obtain a function $f_k$. Defining the superposition $f := \sum_{k \in \K} f_k$, we have from the linearity of $T_0$,
\begin{equation*}\label{eq:po-Tinv}
    T_0f(\sigma) = \ind_{S \in \K} - \mu_0(\K_w)
\end{equation*}
Taking expectations of both sides of \eqref{eq:po-T-diff} with this choice of $f$ establishes the first desired claim:
\begin{equation*}
    \E[c]{(\gamma-\beta)\left(\frac{\Sc-\Sb}{2}\right)\Delta f(S)} = \mu_0(\K_w) - \mu_c(\K_w)
\end{equation*}

We now turn to bounding $\sum_{s = 0}^k |\Delta f(s)|$.
Fix some $k \in \K$ and let $t := (w+k)/2$. Then $k = \cO{\sqrt{w}}$, and so $\left|t \vee (w-t)\right| = \cT{w}$.
By \cref{stein-inv}, part c, we have
\[
    \Delta f_k(t) \le \min\left(a_{t}^{-1},b_{t}^{-1}\right) = \min\left(t^{-1},(w-t)^{-1}\right) = \cO{ \frac{1}{w}}
\]
\Cref{stein-inv}, part d, then implies $\sum_{s}|\Delta f_k(s)| = \cO{w^{-1}}$. Summing over $k \in \K$ yields the second desired result, completing the lemma.
\begin{align*}
    \sum_{s}|\Delta f(s)| &\le \sum_{s}\sum_{k \in \K}|\Delta f_k(s)| = \cO{|\K|w^{-1}}
\end{align*}
\end{proof}

\begin{proof}[Proof of lemma \ref{po-G-tail}:]
Write $\dd = s - w/2$, and let $j = c - \beta s$.
By standard tail bounds for binomial random variables (\cref{cramer}), we have
\begin{align*}
    \binom{(w+k)/2}{\gamma s-j}\beta^{\frac{w+k}{2} - \gamma s +j } \gamma^{\gamma s-j} &= \cO{ \frac{1}{\sqrt{w}}} 
    \cE{- \frac{2 (\gamma \dd - j - \gamma k/2)^2}{w+k}} \\
    \binom{(w-k)/2}{\beta s + j }\beta^{\beta s + j} \gamma^{\frac{w-k}{2}-\beta s - j} &= \cO{ \frac{1}{\sqrt{w}}} \cE{-\frac{2(\beta \dd + j + \beta k/2)^2}{w-k}}
\end{align*}
Multiplying these two inequalities together and using the fact that $k = \cO{w^{1/2}}$ yields
\begin{align*}
G(k, \beta s + j) & = \cO{ \frac{1}{w}} \cE{- \frac{(\dd + x k)^2 + (x\dd + 2j + k)^2}{2w}}
\end{align*}
There are three cases to check.
If $|\dd|, |j| = \cO{w^{1/2}}$, then the quantity in the exponent is of constant order, so the entire expression is $\cO{\frac 1w}$, which agrees with the desired bound.
If $|\dd| = \omega({w^{1/2}})$ and $|\dd| \geq |j|$, then $(\dd + xk)^2 = \cT{\dd^2} = \cT{\dd^2 + j^2}$, which yields the desired bound.
if $|j| = \omega({w^{1/2}})$ and $|j| \geq |\dd|$, then $(x\dd + 2j + k)^2 = \cT{j^2} = \cT{\dd^2 + j^2}$ as well.
All together, we have
\begin{equation*}
G(k, \beta s + j) = \cO{\frac 1w} \cE{-\cT{\frac{\dd^2 + j^2}{w}}}\,,
\end{equation*}
as claimed.
\end{proof}

\begin{proof}[Proof of \cref{G-expansion}]
Rewrite the definition of $G$ as the product of two binomial densities:
\begin{equation*}
    G(k, \beta s + j) := \binom{\frac{w+k}{2}}{\gamma s - j }  \beta^{\frac{w+k}{2}-\gamma s + j} \gamma^{\gamma s - j} \binom{\frac{w-k}{2}}{\beta s + j} \beta^{\beta s + j} \gamma^{\frac{w-k}{2}-\beta s - j} 
\end{equation*}
Let
\[
    \zeta := \frac{\gamma s - j }{(w+k)/2}\,.
\]
Then, $\zeta -\gamma = \cO{(|j| \vee |\dd|)/w} = \oo(1)$. Indeed, 
\begin{equation}\label{eq:po-zeta}
    \zeta - \gamma =  \left( \frac{\dd}{w} - \frac{2(j+x\dd)}{w}\right)  \mO{\frac{k}{w}}
\end{equation}
and this is $\oo(1)$ since we have assumed that $j$ and $\dd$ are both $\cO{\sqrt{w\log w}}$.

Applying \cref{demoivre} and recalling that $k = \cO{w^{1/2}}$, we obtain
\begin{align}
\binom{\frac{w+k}{2}}{\gamma s - j }  \beta^{\frac{w+k}{2}-\gamma s + j} \gamma^{\gamma s - j} & = \mO{\frac 1w} \frac{1}{\sqrt{\pi (w+k) \beta \gamma}} \e^{- \frac{w(\zeta - \gamma)^2}{4 \beta \gamma} + \cO{k (\zeta - \gamma)^2} + \cO{\zeta - \gamma} + \cO{w (\zeta - \gamma)^3}} \nonumber \\
& = \left(1 + \cO{\frac 1w} + \cO{\frac{|j| \vee |\dd|}{w}} + \cO{\frac{(|j| \vee |\dd|)^3}{w^2}}\right) \frac{1}{\sqrt{\pi (w+k) \beta \gamma}} \e^{- \frac{w(\zeta - \gamma)^2}{4 \beta \gamma}}\,.\label{eq:po-binom1-approx}
\end{align}

Similarly, if we let
\begin{equation*}
\eta := \frac{\beta s + j}{(w-k)/2}\,,
\end{equation*}
then an identical computation shows
\begin{equation}\label{eq:po-eta}
\eta - \beta = \left(\frac{\dd}{w} + \frac{2 (j+x\dd)}{w}\right)\mO{\frac k w}
\end{equation}
and
\begin{equation}
 \binom{\frac{w-k}{2}}{\beta s + j} \beta^{\beta s + j} \gamma^{\frac{w-k}{2}-\beta s - j} =  \left(1 + \cO{\frac 1w} + \cO{\frac{|j| \vee |\dd|}{w}} + \cO{\frac{(|j| \vee |\dd|)^3}{w^2}}\right) \frac{1}{\sqrt{\pi(w-k)\beta \gamma}} \e^{-\frac{w(\eta - \beta)^2}{4\beta\gamma} } \label{eq:po-binom2-approx}
\end{equation}

We would like to multiply \eqref{eq:po-binom1-approx} and \eqref{eq:po-binom2-approx} together in order to derive an approximation for $G$.
First, the product of the polynomial prefactors is
\begin{equation*}
\frac{1}{\pi\beta \gamma \sqrt{w^2 - k^2}} = \mO{\frac{k^2}{w^2}} \frac{1}{\pi \beta \gamma} = \mO{\frac 1 w} \frac{1}{\pi w \beta \gamma}\,.
\end{equation*}
Combining~\eqref{eq:po-zeta} and \eqref{eq:po-eta} and using the fact that $k = \cO{\sqrt k}$ and $x \leq w^{-1/2}$ yields
\begin{equation*}
(\zeta -\gamma)^2 + (\eta - \beta)^2 = \mO{\frac k w}\left(\frac{2 \dd^2}{w^2} + \frac{8 (j + x \dd)^2}{w^2}\right) = \mO{w^{-1/2}}\left(\frac{2 \dd^2}{w^2} + \frac{8 j^2}{w^2}\right)\,;
\end{equation*}
therefore, the product of the exponential terms is
\begin{equation*}
\cE{-\frac{2\dd^2 + 4 j^2}{2 \beta \gamma w} \mO{w^{-1/2}}} = \mO{\frac{(|j| \vee |\dd|)^2}{w^{3/2}}} \e^{-\frac{2\dd^2 + 4 j^2}{2 \beta \gamma w}}\,.
\end{equation*}
We conclude that
\begin{align*}
G(k, \beta t + j) & = \left(1 + \cO{\frac 1w} + \cO{\frac{|j| \vee |\dd|}{w}} + \cO{\frac{(|j| \vee |\dd|)^3}{w^2}}+ \cO{\frac{(|j| \vee |\dd|)^2}{w^{3/2}}} \right)\frac{1}{\pi w \beta \gamma}\e^{-\frac{2\dd^2 + 4 j^2}{2 \beta \gamma w}} \\
& = \left(1 + \cO{\frac{1}{w^{-1/2}}} + \cO{\frac{(|j| \vee |\dd|)^3}{w^2}} \right)\frac{1}{\pi w \beta \gamma}\e^{-\frac{2\dd^2 + 4 j^2}{2 \beta \gamma w}}\,,
\end{align*}
as claimed.
\end{proof}

\subsection{Proof of Theorem \ref{main}}
The proof will closely follow the proof of \cref{fullcurve}.
Again by \cref{pz} and \cref{parity_coupling}, the following suffices:
\begin{theorem}\label{main_smm}
    Let $A$ be from the $(m,n,p)$-Bernoulli ensemble with $n \ge C m\log m$, where $C > (2 \log 2)^{-1}$, and let $P$ be the event that each row of $A$ sums to an even number.
    Let $W$ be the vector of row weights.
    Then, for any such $w \in (2 \mathbb N)^m$ with $w_i < .49 n$ for all $i$, %
    \[
        \E{Z_0^2 \big| W=w} = \left(1 + \oo(1)\right) \E{Z_0\big|W=w}^2
    \]
\end{theorem}
To obtain \cref{main}, let $A$ be from the $(m, n, p)$-Bernoulli ensemble, conditioned on the event $P$ that each row of $A$ sums to an event number. By assumption $p \le 1/2$. And, \Cref{mnp} already established that $\disc(A)  = 0$ with high probability if $m = \oo(np)$, so we may assume that $p \le .48$.
Then, with high probability, $W_i < .49 n$ for all $i$ simultaneously.
In other words the collection of events $\{W = w\}$ for $w \in (2 \mathbb N)^m$ with $w_i < .49 n$ satisfy
\begin{equation*}
\sum_w \PP{W = w \mid P} = 1 - \oo(1)\,.
\end{equation*}
Therefore, by \cref{pz}, proving \cref{main_smm} will imply that $\disc(A) = 0$ with high probability.
Finally, removing the conditioning on $P$ by \cref{parity_coupling} proves the claim.

We will check the three conditions of \cref{2mm} for $\K = \{0\}$ when $A$ is from the Bernoulli ensemble conditioned on $W = w$.
Say $u$ and $v$ are balanced $\{-1,+1\}^n$ vectors agreeing on $\beta n$ coordinates, and again suppress the subscript $i$ whenever not ambiguous, e.g. $w = w_i$, $\phi = \phi_i$, and $\psi = \psi_i$. Recalling $\gamma = 1-\beta$, we have
\[
    \psi := \binom{n/2}{w/2}^2\binom{n}{w}^{-1}, \quad \phi(\beta) = \sum_{t = 0}^{w/2} \binom{\beta n/ 2}{t}^2\binom{\gamma n / 2}{w/2-t}^2\binom{n}{w}^{-1}
\]
By Stirling's formula (\cref{stirling}), $\psi = \cT{w^{-1/2}} = \cOm{n^{-1/2}}$. Now consider $\phi(\beta)$ with $\beta \in [\del,1-\del]$ for some universal constant $\del>0$. Since we have assumed that $w_i \le .49n$ for all $i$, the binomial coefficients in the definition of $\phi$ are all non-zero by taking $\del$ sufficiently small. For a sufficiently large constant $M$, standard hypergeometric tail bounds (\cref{hyp_stirling}) yield
\begin{align*}
    \phi(\beta) &= \mO{\frac{1}{w}} \sum_{|t - \beta w/2| \le M\sqrt{w\log w}}  \binom{\beta n/ 2}{t}^2\binom{\gamma n / 2}{w /2 - t}^2\binom{n}{w}^{-1} \\
    &= \mO{\frac{1}{w}} \binom{n/2}{w/2}^{2}\binom{n}{w}^{-1}\sum_{|t - \beta w/2| \le M\sqrt{w\log w}}  \binom{\beta n/ 2}{t}^2\binom{\gamma n / 2}{w /2 - t}^2 \binom{n/2}{w/2}^{-2} \\
    &= \cO{\frac{1}{w^{-3/2}} } \sum_{t \in \ZZ} \e^{-\cT{\frac{t - \beta w/2)^2}{w}}} \\
    & = \cO{w^{-1}}\,.
\end{align*}
where the last step uses \cref{gaussian-sum}.
Since
\begin{equation*}
\E{Z_0\big|W=w} = \binom{n}{n/2} \prod_{i=1}^n \psi_i = \binom{n}{n/2} \originalleft.\cOm{n^{-1/2}}\originalright.^m \geq \cE{n \log 2 - \mo{1}\frac m 2 \log n}\,,
\end{equation*}
the assumption that $n \geq C m \log m$ for $C > (2 \log 2)^{-1}$ implies that the first-moment condition \eqref{eq:2mm-goals-expectation} holds.
Similarly, the weak bound \eqref{eq:2mm-goals-weak} holds: for some implicit constants that depend only on $\del$, 
\[
    \phi(\beta) = \cO{w^{-1}} \leq C_\delta \psi_i^2\,.
\]

It remains to establish the strong inequality \eqref{eq:2mm-goals}. As in the proof of \cref{fullcurve}, Stirling's approximation only allows us to compute $\phi$ and $\psi$ up to an error of $\mO{w^{-1}}$, which we cannot afford. We again use Stein's method to circumvent this challenge.

We adopt a modified version of the construction used in \cref{fullcurve}.
If we condition on the event that $W_i = w$, we can generate $A_i$ by choosing $w$ coordinates uniformly \emph{without replacement} from $\{1, \dots, n\}$ and setting the corresponding coordinates of $A_i$ to~1.
Given two balanced vectors $u$ and $v$, it again suffices to track how many of the chosen coordinates correspond to entries of $u$ and $v$ which are both positive, both negative, or of mixed sign.

Consider a urn of $n$ balls labeled with $(+,+)$, $(+,-)$, $(-,+)$ or $(-,-)$, where we assign $\beta n/2$ to the labels $(+,+)$ and $(-,-)$ and $\gamma n/2$ to $(+, -)$ and $(-, +)$.
We select $w$ balls independently without replacement from this urn, and view each outcome as reflecting the signs of the entry of $v$ and entry of $u$ corresponding to the selected coordinate.
As before, we define the vector $\sigma =(\Sa, \Sb, \Sc, \Sd)$ of counts, and construct an exchangeable copy of $\sigma'$ by choosing one of the $w$ selected balls uniformly at random and swapping it with a random ball in the urn.
The joint law of $(\sigma, \sigma')$ is the \emph{unconditioned distribution}, $\pp_0$.

For the conditioned distribution, we again divide the balls into two types---$\{(+, +), (+, -)\}$ on the one hand, $\{(-, -), (-, +)\}$ on the other---and consider drawing as above $w$ balls without replacement from the urn, but conditioned on the event that exactly $w/2$ balls of each type are chosen.
We obtain a different distribution on count vectors $\sigma$; to construct an exchangeable pair, we generate another vector $\sigma'$ by choosing one of the $w$ selected balls uniformly at random and swapping it with a random ball in the urn of the same type.
This induces a joint law on $(\sigma, \sigma')$ under which these variables are again exchangeable, which we call the \emph{conditioned distribution}, $\pp_c$.

We again focus on $S = S(\sigma) = \Sb + \Sc$.
Adopting the same notation as in the proof of \cref{fullcurve} (note that now $\K = \{0\}$), we have
\begin{align}
\begin{split}\label{eq:ber-stein-implies-quadUB}
    &\pp_0\left[S = w/2\right] = \psi = \pp_0[E_\K], \quad \pp_c\left[S = w/2 \right]\psi = \phi(\beta) 
\end{split}
\end{align}
We also have the following analogue of \cref{stein-inv}.
\begin{lemma}\label{ber-stein-inv}
There is a function $f$ satisfying the identity
\begin{equation}\label{eq:ber-final-bd}
    |\mu_c(w/2) - \mu_0(w/2)| = \left| \E[c]{ \left((\Sc-\Sb)^2 -n (\Sc-\Sb)\left(\beta - \frac{1}{2}\right)\right) \Delta f(S)} \right|
\end{equation}
as well as: 
\begin{equation}\label{eq:ber-Df-L1}
    \sum_{s =0}^w|\Delta f(S)| = \cO{ \frac{1}{nw} }
\end{equation}
\end{lemma}

The proofs of this and all succeeding technical lemmas are deferred to the next section. We turn to bounding \eqref{eq:ber-final-bd}. Just as in \cref{fullcurve}, the main technical difficulty of this theorem is to bound $\EE_C\left[ (\Sc-\Sb) \ind_{S = s}\right]$.
The analogue of \cref{po-with-indicator-bd} is the following estimate.
\begin{proposition}\label{ber-with-indicator-bd}
    Uniformly over $s$,
    \begin{equation*}
        \left|\E[c]{(\Sc-\Sb)\ind_{S = s}}\right| = \cO{xw^{1/2}}\cEws\,.
    \end{equation*}
\end{proposition}

We need a similar result to control the remaining part of \eqref{eq:ber-final-bd}. The proof will follow trivially from the techniques developed in \cref{ber-with-indicator-bd}.

\begin{proposition}\label{ber-with-indicator-bd-square}
Uniformly over all $s \in [w]$,
\[
    \E[c]{ (\Sc-\Sb)^2\ind_{S=s}} = \left(\cO{x^2w^{3/2}} + \cO{w^{1/2}}\right) \cEws\,.
\]
\end{proposition}
Let us first show that together these propositions imply the strong bound before proving them. Combining \cref{ber-with-indicator-bd} and \cref{eq:ber-Df-L1}, we obtain
\begin{align*} %
    (xn)\EE_C\left[ (\Sc-\Sb)\Delta f(S) \right] = \cO{x^2 w^{1/2}n} \sum_{s=0}^w \cEws |\Delta f(s)| &= \cO{x^2w^{-1/2}}\,,
\end{align*}
Similarly, using \cref{ber-with-indicator-bd-square} and recalling $n = \Omega(w)$ yields
\begin{align*} %
    \EE_C\left[ (\Sc-\Sb)^2 \Delta f(S)\right] &= \left(\cO{x^2w^{3/2}} + \cO{w^{1/2}}\right) \sum_{s=0}^w \cEws |\Delta f(s)| \\
    &= \cO{x^2w^{-1/2}} + \cO{n^{-1}w^{-1/2}}
\end{align*}
Applying these to \cref{eq:ber-final-bd}, we obtain:
\begin{align*} 
    \left|\mu_0(w/2) - \mu_c(w/2)\right| \le \cO{x^2w^{-1/2}} + \cO{n^{-1}w^{-1/2}}  %
\end{align*}
The desired strong bound \eqref{eq:2mm-goals} then follows from the identities \eqref{eq:ber-stein-implies-quadUB}. Indeed, under the unconditioned law $\pp_0$, the variables $\Sb$ and $\Sc$ are independent hypergeometric variables
\[ 
    \Sb \sim H\left(\frac{w}{2}; \frac{\gamma n}{2}, \frac{n}{2}\right), \quad \Sc \sim H\left(\frac{w}{2}; \frac{\beta n}{2}, \frac{n}{2}\right)
\]
So, recalling $\psi = \mu_0(w/2) = \cT{w^{-1/2}}$, 
\[
    \phi\left(\frac{1}{2}+x\right) = \psi \mu_c(w/2) = \mO{\frac{1}{n}}\psi^2\mO{x^2}\,.
\]
This verifies the assumptions of Lemma~\ref{2mm} and proves the claim.
We conclude by proving \cref{ber-with-indicator-bd} and \cref{ber-with-indicator-bd-square}.

\begin{proof}[Proof of \cref{ber-with-indicator-bd}]
We proceed identically to the proof of \cref{po-with-indicator-bd}, except with the simplification of only considering $k = 0$. 

Let $F(c) \defeq \pp_c[\Sc = c,~ S = s] = \pp_c[\Sc = c, \Sb = s - c]$.
We can write this explicitly as a product of hypergeometric densities:
\begin{align*}
    F(c) &= \binom{\beta n/2}{c}\binom{\gamma n/2}{w/2-c}\binom{\beta n /2}{w/2 - (s-c)} \binom{\gamma n /2 }{s-c}\binom{n/2}{w/2}^{-2}
\end{align*}
Fix $s$ with $0 \leq s \leq w$. We aim to show the inequality
\begin{align}
    \E[c]{ (\Sc-\Sb) \ind_{S = s}} &=: \sum_{c = 0}^{s}\left(2c - s\right)F(c) \le \cO{xw^{1/2}}\cEws \label{eq:ber-with-indicator-goal}
\end{align}

We begin by giving a crude tail bound on $F$. Using \cref{hyp_stirling} yields
\begin{lemma}\label{ber-F-tail}
For all $c$,
    \begin{equation*}
        F(c) 
        =\cO{ \frac{1}{w}} \cE{-\cT{\frac{(s-w/2)^2 + (c - \beta s)^2}{w}}} 
    \end{equation*}
\end{lemma}
As in \cref{po-with-indicator-bd}, we consider three cases for the size of $|x|$. \\

\paragraph{\textbf{Case 1}: $|x| \geq w^{-1/2}$} Just as in \cref{po-with-indicator-bd}, \cref{ber-F-tail} is already enough when $|x| \geq w^{-1/2}$. The calculation is identical, but with \cref{ber-F-tail} in place of \cref{po-G-tail}, so we omit the details. \\

\paragraph{\textbf{Case 2}: $|x| \leq w^{-1}$}
For the remaining cases, we again exploit the symmetry $c \to s-c$ just as in \cref{po-with-indicator-bd}.
\begin{align}
    \E[c]{ (\Sc-\Sb) \ind_{S = s}} &= \frac{1}{2}\left[\sum_{c = 0}^{s}\left(2c - s\right)F(c) + \sum_{c = 0}^{s}\left(2c - s\right)F(c)\right] \nonumber \\
    &= \frac{1}{2}\sum_{c = 0}^{s}\left(2c - s\right)\left[F(c) -F(s-c)\right] 
    \label{eq:ber-F-diff}
\end{align}

If $x = 0$, then \eqref{eq:ber-F-diff} is identically zero, proving the claim.
So assume that $x \neq 0$, which implies that $|xn|$ is a positive integer.
If $x > 0$, then expanding the definition of $F$ yields
\begin{align*}
    \frac{F(c)}{F(s-c)} &= \frac{\binom{\beta n/2}{c}\binom{\gamma n/2}{w/2-c}\binom{\beta n /2}{w/2 - (s-c)} \binom{\gamma n /2 }{s-c}}{\binom{\gamma n/2}{c}\binom{\beta n/2}{w/2-c}\binom{\gamma n /2}{w/2 - (s-c)} \binom{\beta n /2 }{s-c}} \\
    &= \prod_{r = -xn+1}^{xn}\left(\frac{\frac{n}{4} - \frac{w}{2} + c + r}{ \frac{n}{4} -  c + r   } \right)\left(\frac{\frac{n}{4} - s + c + r}{\frac{n}{4} - \left(\frac{w}{2} - s + c\right) + r}\right) \\
    &= \prod_{r = -xn+1}^{xn}\left(1 + \cO{\frac{2c - w/2}{n}} \right) \left(1 + \cO{\frac{w/2  + 2c - 2s}{n}} \right) \\ 
    &= \originalleft(1 + \cO{\frac{2c-w/2}{n}} + \cO{\frac{w/2-2(s-c)}{n}}\originalright)^{\cT{xn}}
\end{align*}  
If $x < 0$, then the numerator and denominator of the product in the second line are flipped, but the third line is reached unchanged. So, regardless of the sign of $x$, this computation holds. Then, recalling $|x| \le w^{-1}$, we are justified in using a first-order Taylor expansion:
\[
    \frac{F(c)}{F(s-c)} = 1 + \cO{x(2c-w/2)} + \cO{x(w/2-2(s-c))} = 1 + \cO{x^2s} + \cO{x(c - \beta s)} + \cO{x(w/2-s)}\,,
\]
where the second equality follows from the fact that
\begin{equation*}
|2c - w/2| \leq |w/2 - s| + 2|c - \beta s| + 2|xs|\,,
\end{equation*}
and the analogous bound for $|w/2 - 2(s-c)|$.
Since $|x| \leq w^{-1}$, we have $x^2 s = \cO{x}$ and $|2c - s| \leq 2|c - \beta s| + 1$.
Applying \cref{ber-F-tail,gaussian-sum}, we obtain
\begin{align*}
    \sum_{c = 0}^s (2c - s) (F(c) - F(s-c)) &= \cO{\frac x w} \sum_{c = 0}^s \left(\cO{(c-\beta s)^2} + \cO{(w/2-s)^2} \right) \e^{-\cT{\frac{(c - \beta s)^2 + (s - w/2)^2}{w}}}\\
    & = \left(\cO{x w^{1/2}} + \cO{x w^{-1/2} (s - w/2)^2}\right) \cEws \\
    & = \cO{x w^{1/2}}\cEws\,,
\end{align*}
where the last step uses the fact that, for any positive constant $C$,
\begin{equation*}
y^2 \e^{- C y^2/w} = \cO{w} \e^{- C y^2/2w}\,.
\end{equation*}
This finishes the proof of the second case. \\

\paragraph{\textbf{Case 3}: $w^{-1} \le |x| \le w^{-1/2}$}
Again our plan is to derive a careful estimate for $F(c)$ near $c = \beta s$ via a Taylor expansion of Stirling's formula. Because of the tail estimate \cref{ber-F-tail}, as in the proof of \cref{po-with-indicator-bd} we may again restrict to the region where $|c - \beta s| = \cO{\sqrt{s\log s}}$ and $|s -w/2| = \cO{\sqrt{s\log s}}$.

Following precisely the steps of the proof of Case 3 in \cref{po-with-indicator-bd}, using the tail estimate \cref{ber-F-tail} in place of \cref{po-G-tail}, we obtain that
\begin{equation}\label{eq:ber-case3-goal}
\E[c]{(\Sc-\Sb)\ind_{S = s}} = \sum_{|c - \beta s| = \cO{\sqrt{w \log w}}} (2c - 2\beta s) (F(c) - F(s+2[xs] - c)) + \cO{xw^{1/2}}\cEws\,,
\end{equation}

The only difference from \cref{po-with-indicator-bd} is that we are dealing with a product $F$ of hypergeometric distributions rather than the product $G$ of binomial distributions. However, $F$ and $G$ are conveniently related because they describe hypergeometric and binomial distributions respectively with the same mean and number of trials. Denote $G(c) := G(0,c)$. An explicit expansion of the hypergeometric density~\cite{pinskynormal} yields:
\begin{align}
    F(c) &= G(c)R_\beta(c)R_\gamma(s-c),\label{eq:ber-G-F} \\ 
    R_\beta(c) &:= \frac{\prod_{t = 1}^{c-1}\left(1 - \frac{t}{\beta n /2}\right) \prod_{t = 1}^{w/2-c-1}\left(1 - \frac{t}{\gamma n /2}\right)}{\prod_{t = 1}^{w/2-1}\left(1 - \frac{t}{n/2}\right)} \nonumber
\end{align}
Let us write $\tilde s = s + 2[xs]$. 
Since $|G(c) - G(\tilde s-c)|$ was already bounded in \cref{po-with-indicator-bd}, it suffices to bound $R_\beta(c)R_\gamma( s-c)-R_\beta(\ts-c)R_\gamma(s - \tilde s + c) $.

We temporarily define an unorthodox convention for the product that will save us much case work. If $a < b$, then define $\prod_{i=a}^b x_i = x_ax_{a+1}...x_b$ as usual. However, if $a > b$, we define $\prod_{i=a}^b x_i = \prod_{i=b+1}^{a-1}x_i^{-1}$. Then, expanding the definition of $R$,  
\begin{align*}
    \frac{R_\beta(c)}{R_\gamma(s-\ts+c)} &= \frac{\prod_{t=1}^{c-1}\left(1 - \frac{t}{\beta n /2}\right) \prod_{t = 1}^{w/2-c-1}\left( 1 - \frac{t}{\gamma n /2}\right)} { \prod_{t = 1}^{s-\ts + c -1}\left(1 - \frac{t}{\gamma n /2}\right) \prod_{t=1}^{w/2-s+\ts-c-1}\left(1 - \frac{t}{\beta n /2}\right)} \\
    &= \prod_{t = 1}^{s-\ts + c -1}\left(\frac{1 - \frac{t}{\beta n /2}}{1 - \frac{t}{\gamma n /2}}\right) \prod_{t = s-\ts+c}^{c-1} \left(1 - \frac{t}{\beta n /2}\right)  \prod_{t = 1}^{w/2-c -1}\left(\frac{1 - \frac{t}{\gamma n /2}}{1 - \frac{t}{\beta n /2}}\right) \prod_{t = w/2-c}^{w/2-s+\ts-c-1} \left(1 - \frac{t}{\gamma n /2}\right)
\end{align*}
A similar expansion yields:
\begin{align*}
    \frac{R_\gamma(s-c)}{R_\beta(\ts-c)} &= \prod_{t = 1}^{\ts - c -1}\left(\frac{1 - \frac{t}{\gamma n /2}}{1 - \frac{t}{\beta n /2}}\right) \prod_{t = \ts-c}^{s-c-1} \left(1 - \frac{t}{\gamma n /2}\right)  \prod_{t = 1}^{w/2-\ts+c -1}\left(\frac{1 - \frac{t}{\beta n /2}}{1 - \frac{t}{\gamma n /2}}\right) \prod_{t = w/2-\ts+c}^{w/2-s+c-1} \left(1 - \frac{t}{\beta n /2}\right)
\end{align*}
In total, our goal is to bound how far the following expression is from one.
\begin{align}
    \frac{R_\beta(c)R_\gamma(s-c)}{R_\gamma(s-\ts+c)R_\beta(\ts-c)} &= \prod_{t = \ts-c}^{s-\ts + c -1}\left(\frac{1 - \frac{t}{\beta n /2}}{1 - \frac{t}{\gamma n /2}}\right) \prod_{t = w/2-\ts+c}^{w/2-c -1}\left(\frac{1 - \frac{t}{\gamma n /2}}{1 - \frac{t}{\beta n /2}}\right) \prod_{t = s-\ts+c}^{c-1} \left(1 - \frac{t}{\beta n /2}\right) \label{eq:ber-R-goal} \\
    &\quad \times  \prod_{t = w/2-c}^{w/2-s+\ts-c-1} \left(1 - \frac{t}{\gamma n /2}\right) \prod_{t = \ts-c}^{s-c-1} \left(1 - \frac{t}{\gamma n /2}\right) \prod_{t = w/2-\ts+c}^{w/2-s+c-1} \left(1 - \frac{t}{\beta n /2}\right) \nonumber
\end{align}
We bound this expression in several parts. Throughout, we will repeatedly use the arithmetic fact that $(xw)^2 + |xw(c-\beta s)| = \cO{(xw)^2 + (c-\beta s)^2}$. Observe that $2c-\ts + s-\ts = \cO{xs} + \cO{c -\beta s}$. Noting $xt/n = \cO{w^{1/2}/n} = \oo(1)$, we may apply a first-order Taylor expansion:
\begin{align}
    \prod_{t = \ts-c}^{s-\ts + c -1}\left(\frac{1 - \frac{t}{\beta n /2}}{1 - \frac{t}{\gamma n /2}}\right) \prod_{t = w/2-\ts+c}^{w/2-c -1}\left(\frac{1 - \frac{t}{\gamma n /2}}{1 - \frac{t}{\beta n /2}}\right) &= \prod_{t = \ts-c}^{s-\ts + c -1} \left(1 + \cO{\frac{xt}{n}}\right) \prod_{t=w/2-\ts+c}^{w/2-c-1} \left(1 + \cO{\frac{xt}{n}}\right) \nonumber \\
    &= \cO{\cE{\cO{\frac{xw(|xs| + |c -\beta s|)}{n}}}} \nonumber \\
    &= 1 + \cO{\frac{x^2w^2 + (c-\beta s)^2}{n}}\cE{\frac{xw(c-\beta s)}{n}}\label{eq:ber-R-1}
\end{align}
In the last line, we used that $x^2w^2/n = \cO{1}$ as well as the inequality $e^y \le 1 + ye^y$ for all $y \in \RR$.
Next, by our product notation convention,
\begin{align*}
    \prod_{t=s-\ts+c}^{c-1} \left(1 - \frac{t}{\beta n /2}\right)  \prod_{t=\ts-c}^{s-c-1} \left(1 - \frac{t}{\gamma n /2}\right)
    &= \prod_{t=s-\ts+c}^{c-1} \left( \frac{1 - \frac{t}{\beta n /2}}{1 - \frac{t + 2c-\ts}{\gamma n /2}} \right) \\
    &= \prod_{t=s-c}^{c-1} \left(1 + \cO{\frac{xt}{n}} + \cO{\frac{2c-\ts}{n}}\right)
\end{align*}
Both $xt/n = \oo(1)$ and $(2c-s)/n = \oo(1)$, so we may apply a first-order Taylor expansion again. Since $|s-\ts| = \cO{xw}$,
\begin{align} 
    \prod_{t = s-\ts+c}^{c-1} \left(1 - \frac{t}{\beta n /2}\right) \prod_{t = \ts-c}^{s-c-1} \left(1 - \frac{t}{\gamma n /2}\right) &= 1 + \cO{\frac{xw(2c-s)}{n}} + \cO{\frac{(2c-s)^2}{n}} \nonumber \\
    &= 1 + \cO{\frac{x^2w^2 + (c-\beta s)^2}{n}} \label{eq:ber-R-2}
\end{align} 
The last line follows from observing $2c-s= 2(c-\beta s + xs)$, and hence $2c-s = \cO{c-\beta s} + \cO{xw}$. An identical computation also yields control over the remaining products in \eqref{eq:ber-R-goal}.
\begin{equation} \label{eq:ber-R-3}
    \prod_{t = w/2-c}^{w/2-s+\ts-c-1} \left(1 - \frac{t}{\gamma n /2}\right) \prod_{t = w/2-\ts+c}^{w/2-s+c-1} \left(1 - \frac{t}{\beta n /2}\right) = 1 + \cO{\frac{x^2w^2 + (c-\beta s)^2}{n}}
\end{equation}
Hence the contribution of \eqref{eq:ber-R-1} dominates those of \eqref{eq:ber-R-2} and \eqref{eq:ber-R-3}. Returning to \eqref{eq:ber-R-goal},
\begin{equation}\label{eq:ber-R-final}
    \frac{R_\beta(c)R_\gamma(\ts-c)}{R_\beta(s-c)R_\gamma(c+s-\ts)} = 1 + \cO{\frac{x^2w^2 + (c-\beta s)^2}{n}}\cE{\frac{xw(c-\beta s)}{n}}
\end{equation}

This bound is sufficient to control $R$. Next, recall that by \cref{G-expansion} from \cref{po-with-indicator-bd}, 
\[
    |G(c) - G(\ts-c)| \le G(c)E(j,\dd)\,,
\]
where we have set $\dd = s - w/2$ and $j = c - \beta s$.
Returning to \eqref{eq:ber-G-F} with this fact and \eqref{eq:ber-R-final}, and then using the tail bound on $F$ given in \cref{ber-F-tail}, we have:
\begin{align*}
    \left|F(c) - F(\ts-c)\right| &\le F(c) \left(E(j,\dd) + \cO{\frac{x^2w^2 + j^2}{n}}e^{\frac{xwj}{n}} \right) \\ 
    &= \cO{\frac{1}{w}}\cE{-\cT{\frac{j^2 + \dd^2}{w}}} \left(E(j,\dd) + \cO{\frac{x^2w^2 + j^2}{n}}e^{\frac{xwj}{n}} \right)
\end{align*}
Apart from the $\frac{x^2w^2 + j^2}{n}e^{\frac{xwj}{n}}$ term, this is same bound obtained as equation \eqref{eq:po-G-diff-final}  in \cref{po-with-indicator-bd}. From here, the proof in \cref{po-with-indicator-bd} may be followed exactly to see:
\begin{align}
    \sum_c (2c-s)|F(c)-F(s-c)| &= \cO{xw^{1/2}}e^{-\cT{\frac{\dd^2}{w}}}  \label{eq:ber-F-diff-sum} \\ &\quad+  \sum_c \cO{\frac{(|xs| + |j|)(x^2w^2 + j^2)}{nw}}e^{\cO{\frac{xwj}{n}}} e^{-\cT{\frac{j^2 + \dd^2}{w}}} \nonumber
\end{align}
Let us simplify the exponents. Recalling $x = \cO{w^{-1/2}}$ and $n = \cOm{w}$, we have $xwj/n = \cO{j/\sqrt{w}}$. Fixing $b>0$ large enough, if $j \ge b\sqrt{w}$, then $e^{\cO{\frac{xwj}{n}}} e^{-\cT{\frac{j^2 + \dd^2}{w}}} = e^{-\cT{\frac{j^2 + \dd^2}{w}}}$
And, if $j \le b \sqrt{w}$, then $e^{\cO{\frac{xwj}{n}}} = \cO{1}$. So, in total, we have by \cref{gaussian-sum}, 
\begin{align*}
    \sum_c \cO{\frac{(|xs| + |j|)(x^2w^2 + |xwj|)}{nw}}e^{\cO{\frac{xwj}{n}}} e^{-\cT{\frac{j^2 + \dd^2}{w}}} &= \sum_c \cO{\frac{(|xs| + |j|)(x^2w^2 + |xwj|)}{nw}}e^{-\cT{\frac{j^2 + \dd^2}{w}}} \\
    &= \cO{xw^{1/2}} e^{-\cT{\frac{\dd^2}{w}}}
\end{align*}
Combining this with \eqref{eq:ber-F-diff-sum} and \eqref{eq:ber-case3-goal}, we are done. This completes the proposition for all possible $x$.
\end{proof} 

\begin{proof}[Proof of \cref{ber-with-indicator-bd-square}] 
Our goal is to bound:
\begin{equation*}
    \E[c]{ (\Sb-\Sc)^2\ind_{S=s}} = \sum_{c =0}^s (2c - s)^2 F(c) = 4\sum_{c =0}^s (c - \beta s + x s)^2 F(c) = \cO{\sum_{c =0}^s (x^2s^2+(c - \beta s)^2) F(c)}
\end{equation*}
We can directly apply the crude tailbound \cref{ber-F-tail} on $F$ to conclude. 
By \cref{gaussian-sum},
\begin{align*}
    \E[c]{ (\Sb-\Sc)^2\ind_{S=s}} &\le \cO{\frac{1}{w}}\cEws \sum_{c=0}^s 
    (x^2s^2+(c - \beta s)^2) e^{-\cT{\frac{(c - \beta s)^2}{w}}} \\
    &= \left(\cO{x^2w^{3/2}} + \cO{w^{1/2}}\right) \cEws\,.
\end{align*}
\end{proof}

\subsection{Proofs of lemmas}

\begin{proof}[Proof of \cref{ber-stein-inv}] The proof is almost identical to that of \cref{po-f-bound}.
Let $S :=S(\sigma)$ and $S' := S(\sigma')$. A simple computation yields:
\begin{align}
\begin{split}\label{eq:ber-alpha,beta}
    \pp_0\left[S' > S ~|~S\right]
    &= \frac{1}{w(n-w)} \left(\frac{wn}{2}-S \left(w + \frac{n}{2}\right) + S^2\right) =: \frac{a_S}{w(n-w)} \\
    \pp_0\left[S' < S ~|~S\right] &=  \frac{1}{w(n-w)}\left(S\left(\frac{n}{2} - w\right) + S^2\right) =: \frac{b_S}{w(n-w)}
\end{split}
\end{align}
Similarly,
\begin{align*}
    \pp_c\left[S>S' ~|~\Sb,~\Sc \right] &= \frac{2}{w(n-w)}\left((\Sb)^2 + (\Sc)^2 -(\Sb+\Sc)\frac{w}{2}+ \Sc\left(\frac{n}{2}-r\right) + \Sb r\right) \\
    \pp_c\left[S<S' ~|~\Sb,~\Sc\right] &= \frac{2}{w(n-w)}\left((\Sb)^2 + (\Sc)^2 + \frac{wn}{4} -(\Sb+\Sc)\frac{w}{2} - \Sb\left(\frac{n}{2}-r\right) -\Sc r\right)
\end{align*}
Recall the definitions of the skew-symmetric operator $\antisym$ and $T_0$ given in \eqref{eq:def-stein-op}. Then
\begin{equation}\label{eq:T0}
    T_0 f(S) = a_S f(S+1) -b_S f(S)
\end{equation}
Recall $\Delta f(S)$ is the one-step forward difference operator. For any function $f$ that depends only on $S$,
\begin{equation} \label{eq:T-diff}
    T_c f(\sigma) - T_0 f(\sigma)= \left[(\Sb-\Sc)^2 -n (\Sc-\Sb)\left(\beta - \frac{1}{2}\right)\right] \Delta f(S)
\end{equation}
We aim to find an $f$ such that $T_0f = \ind(S = w/2) - \mu_0(w/2)$. Since $S(\sigma)$ under $\mu_0$ is a hypergeometric random variable with $w$ trials, success population $n/2$ and total population $n$, it is easy to check that taking $\mu := \mu_0(S)$, $a_S$ and $b_S$ as given in \eqref{eq:ber-alpha,beta}, and $T$ as in $\eqref{eq:T0}$ satisfies \eqref{eq:lemma-stein-inv-condit}. Thus, we may apply Lemma \ref{stein-inv} with $t := w/2$ to obtain such an $f$. Taking the expectation of both sides of \eqref{eq:T-diff} with this choice of $f$ yields the first desired claim, \eqref{eq:ber-final-bd}. 

Next, we have by definition of $a$ and $b$, as well as the assumption that $n-w=\cT{n}$,
\[
    \Delta f(w/2) \le \min\left(a_{w/2}^{-1},b_{w/2}^{-1}\right) = \frac{4}{(n-w)w} = \cO{ \frac{1}{nw}}
\]
The second desired claim, \eqref{eq:ber-Df-L1}, then directly follows from the third and fourth parts of \cref{stein-inv}:
\[
    \sum_s |\Delta f(w/2)| = \cO{\frac{1}{nw}}
\]
\end{proof}

\begin{proof}[Proof of \cref{ber-F-tail}]
Except for two changes, the proof is the same as that of \cref{po-G-tail}. First, we now have the simplification $k = 0$. Second, we use \cref{hyp_stirling} in place of \cref{cramer}.
\end{proof}

\appendix
\section{Omitted proofs}\label{proofs}

\subsection{Proof of \cref{parity_coupling} }
Let $\mu$ be either the Poisson or Binomial distribution on $\ZZ$. We will construct a variable $X \sim \mu$ as well as another random variable $X'$ satisfying $|X-X'| \le 1$ almost surely, such that $X'$ has the same distribution as $X$ conditioned on the event that $X$ is even. If we can construct such a coupling $(X,X')$, then we can couple a matrix $A$ (from either the Bernoulli or Poisson ensemble) with $A'$ from the same ensemble conditioned on $P$ (the event of even row parities), such that each row of $A$ and $A'$ differ in at most one entry.

To construct this coupling, consider a matrix $A$ drawn from either the Poisson on Bernouli ensemble and write $X_i$, $i = 1, \dots, m$ for its row sums. For each $i$, use the above coupling to construct an $X_i'$ with $|X_i - X_i'| \le 1$ and $X_i'$ even. For $A$ drawn from the Bernoulli ensemble, copy the rows of $A$ to $A'$ with the following modification: if $X_i' = X_i-1$, flip a one to a zero uniformly at random from row $i$. If $X_i' = X_i+1$, then flip a zero to a one uniformly at random. For $A$ drawn from the Poisson ensemble: if $X_i' = X_i-1$, decrement a non-zero entry of the row $i$, with entry $A_{ij}$ picked with probability proportional to $A_{ij}$. If $X_i' = X_i+1$, increment an entry of row $i$ uniformly at random.

It remains to construct the desired coupling $(X, X')$. We use a construction due to Pinelis~\cite{pinelis}. Denote by $\mu'$ the distribution induced by conditioning $\mu$ on being even, and define the shorthand $\mu_i := \mu(i)$, $\mu_i' := \mu'(i)$. The coupling is defined as follows:
\begin{align*}
    \PP{X = 2j, X' = 2j} &= \mu_{2j} \\
    \PP{X = 2j-1, X' = 2j} &= t_{2j} \\
    \PP{X = 2j-1, X' = 2(j-1)} &= \mu_{2j-1} - t_{2j}
\end{align*}
We claim that we can define $t$ to satisfy the following conditions for all $j \in \ZZ \cap [-1,\infty]$:
\begin{align}
    &\mu_{2j}' = \mu_{2j} + t_{2j} + \mu_{2j+1} - t_{2j+2} \label{eq:marg} \\
    &0 \le t_{2j} \le \mu_{2j-1}\label{eq:prob}
\end{align}
If such a $t$ exists, then $(X,X')$ is clearly the desired coupling, so let us construct it.
First, let $Q$ be the probability that $X \sim \mu$ is even.
This probability is positive if $\mu$ is a Poisson distribution.
In the Binomial case, since we have assumed that $p \leq 1/2$ in the definition of the Bernoulli ensemble, we also have $Q > 0$.
If $Q = 1$, then $X$ is even almost surely and there is nothing to show, so we assume in what follows that $Q \in (0, 1)$.

Let $Q$ be the probability that $X \sim \mu$ is even; then $\mu_{2j}/Q = \mu_{2j}'$. Since both the binomial and Poisson distribution are supported on nonnegative integers, we will set $t_{j} = 0$ for all $j \le -2$. Now we construct $t$ to exactly satisfy \eqref{eq:marg}.
\[
    t_{2j+2} = t_{2j} + \mu_{2j}\left(1 - \frac{1}{Q}\right) + \mu_{2j+1}
\]
It remains to check that this construction satisfies \eqref{eq:prob}. We check the two inequalities individually. First,
\begin{align*}
    t_{2j} &= \sum_{i=-2}^{j-1} t_{2(i+1)}-t_{2i} = \sum_{i=-1}^{j-1} \mu_{2j}\left(1 - \frac{1}{Q}\right) + \mu_{2j+1} 
\end{align*}
From the expression we have just derived for $t$, it is clear that $t_{2j} \ge t_{2(j-1)}$ if and only if $\frac{\mu_{2j-1}}{\mu_{2j-2}} > \frac{1}{Q}-1$. For both the Poisson and Binomial cases, the sequence $\mu$ is log-concave:
\begin{equation*}
\mu_j^2 \geq \mu_{j-1} \mu_{j+1} \quad \forall j \in \ZZ\,,
\end{equation*}
which implies that the sequence of ratios $\frac{\mu_{2j-1}}{\mu_{2j-2}}$ is decreasing.
Therefore, the sequence $t_{2j}$ changes from increasing to decreasing at most once.
Since $t_{-1} = 0 = t_{\infty}$ and $t_0 > 0$, this implies that $t$ can never be negative.
Similarly, we can write a telescoping sum to check the other condition:
\begin{align*}
    t_{2j} - \mu_{2j-1} &= \sum_{i=-1}^{j-1} \mu_{2j}\left(1 - \frac{1}{Q}\right) + \mu_{2j-1} 
\end{align*}
By identical reasoning, this expression is always non-negative, so indeed \eqref{eq:prob} is satisfied. This completes the lemma. \QEDA \\

\subsection{Laplace's Method: Proof of Lemma \ref{2mm}}
By assumption, the columns of $A \sim \M$ are exchangeable, so $\psi_i$ and $\phi_i$ do not depend on particular choices of $u$ and $v$. We also assumed the rows of $A$ are independent, so the expectation of $Z$ is given by
\begin{equation}
    \E{Z}^2 = \left(\sum_{u \in \mathcal B} \PP{ u \in \bigcap_{i=1}^m G_i} \right)^2 = \binom{n}{n/2}^2 \prod_{i=1}^m \psi_i^2
\end{equation}

Turning to the second moment, the number of pairs of balanced vectors that are equal in $2r$ indices and different in the remaining $n/2-2r$ indices is:
\[
    \binom{n}{n/2}\binom{n/2}{r}^2
\]
The first coefficient gives the number of ways to pick a balanced vector $u$; then, looking at the $n/2$ coordinates each where $u$ is $+1$ and $-1$ we must pick $r$ coordinates from each for $v$ to agree with $u$ and have $v$ differ on the remaining coordinates.

We will break the second moment calculation into three regimes. Consider $r \in [n]$ and define $\beta = 2r/n$.
Let $\delta \in (0, 1/2)$ be small enough that $2 H(2\delta) \leq c$, where $c$ is the constant appearing in the first-moment bound~\eqref{eq:2mm-goals-expectation}.
Let $I_1 := \{r:~\beta \in (1/2-\eps, 1/2+\eps)\}$ be the central region, $I_2 := \{r:~ \beta \in [\del, 1-\del ] \setminus I_1\}$ be an annulus, and $I_3 := [n] \setminus (I_1 \cup I_2)\}$ be the remainder. Then:    
\begin{align*}
    \E{Z^2} &= \binom{n}{n/2}\sum_{r = 0}^{n/2}\binom{n/2}{r}^2 \prod_{i=1}^m \phi_{i}\left(\frac{2r}{n}\right) = \binom{n}{n/2} \sum_{r \in I_1 \sqcup I_2 \sqcup I_3}\binom{n/2}{r}^2 \prod_{i=1}^m \phi_{i}\left(\frac{2r}{n}\right)
\end{align*}
We begin by showing the annular region $I_2$ is negligible. Recall $\beta := 2r/n$ and let $x := \beta - 1/2$. Using the weak bound~\eqref{eq:2mm-goals-weak} and standard tail bounds on binomial coefficients (\cref{cramer}) yields
\begin{align}
    \binom{n}{n/2} \sum_{I_2} \binom{n/2}{\beta n /2}^2 \prod_{i=1}^m \phi_{i}\left(\beta\right) &\le  \binom{n}{n/2} \sum_{I_2} \binom{n/2}{n/4}^2 \e^{- \cT{x^2 n}} C_\delta^m \prod_{i=1}^m \psi_i^2 \nonumber \\
    &\le  \E{Z}^2 \cE{-\cT{n} +\cT{m}} \nonumber \\
    & = \E{Z}^2 \cE{-\cT{n}} \label{eq:2mm-I2}
\end{align}
The second line follows from the fact that $\beta$ is bounded away from $1/2$ on $I_2$, and the third line uses $m = \oo(n)$.
Now, we show the outermost region is also negligible. Using the trivial bound that $\phi_i(\beta) \le \psi_i$ for all $\beta$,
\begin{align}
    \binom{n}{n/2} \sum_{I_3}\binom{n/2}{r}^2 \prod_{i=1}^m \phi_{i}\left(\frac{2r}{n}\right) &\le n \E{Z} \binom{n/2}{\del n}^2 \nonumber\\
    &\le \E{Z} \e^{\mo{1}nH(2\del)} \nonumber \\
    & \le \E{Z}^2 \e^{-\mo{1}nH(2\del)}\,,\label{eq:2mm-I3}
\end{align}
where the final inequality uses the assumption that $\EE[Z] \geq \e^{c n} \geq \e^{2 n H(2\delta)}$.
Hence the contribution of $I_3$ is also exponentially small.

So, $I_1$ is the only region with meaningful contribution. The sum over $I_1$ can be understood through Laplace's method: away from $r = n/4$, the term $\binom{n/2}{r}$ decays exponentially in $n$ and $\prod_{i=1}^n \phi_i$ grows at most exponentially in $m$.
In order for the central term when $r = n/4$ to agree with $\EE[Z]^2$ to first order, we therefore need that $\phi_i(1/2) = (1 + \oo\left(m^{-1}\right))\psi_i^2$, and to ensure that the central term is the dominant one we need to show that $\phi_i(1/2+x)$ grows at most quadratically in a window of constant radius around $x = 0$.
This is the purpose of the inequality  \eqref{eq:2mm-goals}. Using \eqref{eq:2mm-I2} and \eqref{eq:2mm-I3} to ignore $I_2$ and $I_3$, and then applying \eqref{eq:2mm-goals} to bound $\phi$ in $I_1$, 
\begin{align}
    \E{Z^2} &\le \mo{1}\binom{n}{n/2}\sum_{r \in I_1} \binom{n/2}{r}^2 \prod_{i=1}^m \mo{\frac 1m}  \left(1 + C \left(\frac{2r}{n}-\frac{1}{2}\right)^2\right) \psi_i^2\nonumber \\
    &\le \mo{1} \binom{n}{n/2}^{-1}\E{Z}^2 \sum_{I_1}\binom{n/2}{r}^2 \exp \left( m C\left(\frac{2r}{n} - \frac{1}{2}\right)^2\right)\label{eq:2mm-I1}
\end{align}

Recall $m = \oo(n)$ and $I_1 = \{r:~|r-n/4| < \eps n\}$.
Using Stirling's approximation (\cref{stirling}), we obtain
\begin{align*}
    \sum_{|r-n/4| \le \eps n}\binom{n/2}{r}^2 e^{m C\left(\frac{2r}{n} - \frac{1}{2}\right)^2}
    &= \mo{1}\sum_{|r-n/4| \le \sqrt{n\log n}}\binom{n/2}{r}^2 \e^{m C\left(\frac{2r}{n} - \frac{1}{2}\right)^2} \\
    & = \mo{1}\frac{2^{n+2}}{n\pi} \sum_{|r-n/4| \le \sqrt{n\log n}} \e^{-\frac{8(r-n/4)^2}{n} + m C\left(\frac{2r}{n} - \frac{1}{2}\right)^2} \\
    & = \mo{1}\frac{2^{n+2}}{n\pi} \sum_{j \in \mathbb Z}\e^{-\frac{8j^2}{n}(1 + \cO{\frac{m}{n}})}\\
    & = \left(1 + \oo(1) + \cO{\frac m n}\right) \sqrt{\frac{2}{\pi n}} 2^n \sum_{j \in \mathbb Z} \e^{-\frac{n \pi^2 j^2}{8}(1 + \cO{\frac{m}{n}})}\,,
\end{align*}
where in the last line we have used the functional equation for the Jacobi theta function~\cite[Theorem 3.2]{SteSha03}.
Since $\sqrt{\frac{2}{\pi n}} 2^n = \mo{1} \binom{n}{n/2}$ and $\sum_{j \in \mathbb Z} \e^{-\frac{n \pi^2 j^2}{8}(1 + \cO{\frac{m}{n}})} = 1 + \oo(1)$, we have shown
\begin{equation*}
    \sum_{|r-n/4| \le \eps n}\binom{n/2}{r}^2 e^{m C\left(\frac{2r}{n} - \frac{1}{2}\right)^2} = \mo{1}\binom{n}{n/2}\,.
\end{equation*}
Returning to \eqref{eq:2mm-I1} and applying this bound for $I_1$, the lemma is established: $\E{Z^2} \le \left(1 + \oo(1)\right) \E{Z}^2$.
\QEDA \\

\subsection{Proof of \cref{stein-inv}}
Denote by $U_{s}$ the integers $\{0,...,s\}$ and use the shorthand $\mu_i := \mu(\{i\})$. Define $f$ by
\begin{equation}\label{eq:f-formula}
    f(0) = 0\,, \quad f(s) := \frac{\mu_t}{b_s\mu_s}\left(\ind_{t<s} - \mu(U_{s-1})\right) \quad \forall s \in \{1, \dots, w\}\,.
\end{equation}
Then, since $\mu_{s+1} = \mu_s a_s / b_{s+1}$, we have for all $s \in [w]$,
\begin{align*}
    a_s f(s+1) - b_s f(s) &= \mu_t \left(\frac{a_s}{b_{s+1}\mu_{s+1}}\ind_{t < s+1 } - \frac{1}{\mu_s} \ind_{t < s } - \frac{a_s}{b_{s+1}\mu_{s+1}}\mu(U_{s}) + \frac{1}{\mu_s} \mu(U_{s-1})\right)\\
    &= \frac{\mu_t}{\mu_s} \left[\left(\ind_{t < s+1 }- \ind_{t < s }\right) - \left(\mu(U_{s}) -\mu(U_{s-1}) \right)\right] \\
    &= \ind_{s = t} - \mu_t
\end{align*}
This establishes claim (a). Now we prove (b). Note that $a_s, b_s, \mu_s \ge 0$ for all $s$. So, if $s \le t$, then $f(s) = -\frac{\mu_t}{b_s\mu_s} \mu(U_{s-1}) \le 0 $. And, if $s \ge t+1$, then $f(s) = \frac{\mu_t}{b_s\mu_s} (1-\mu(U_{s-1})) \ge 0 $ since $\mu \le 1$. This establishes the sign of $f$. Next is monotonicity.
Since $t \geq 1$, we have $f(1) \leq 0 = f(0)$, so $f$ is non-increasing from $f(0)$ to $f(1)$.
Now, say $2 \leq s \le t$. We have:
\begin{align*}
    f(s-1) - f(s) = \mu_t\left(\frac{\mu(U_{s-1})}{b_{s}\mu_{s}} - \frac{\mu(U_{s-2})}{b_{s-1}\mu_{s-1}}\right) \ge \frac{\mu_t}{b_{s}\mu_{s}} \left(\sum_{i=1}^{s-1} \mu_i \left(1  - \frac{b_s \mu_s \mu_{i-1}}{b_{s-1}\mu_{s-1}\mu_i}\right) \right) \ge 0
\end{align*}
The last inequality follows from $i \le s-1 < t$ and the fact that $a$ is decreasing and $b$ is increasing:
\[
    \frac{b_s \mu_s \mu_{i-1}}{b_{s-1}\mu_{s-1}\mu_i} = \frac{a_{s-1}b_i}{a_{i-1}b_{s-1}} \le 1
\]
The argument for $s \ge t+1$ is almost identical and yields the reverse inequality. Now, we prove (c). By (b), we already know $\sup |\Delta f(s)| \le \Delta f(t)$. So, we just need to show that $\Delta f(t) \le \min\left(a_t^{-1}, b_t^{-1}\right)$.
\begin{align*}
    \Delta f(t) = a_t^{-1} \sum_{i = t+1}^w \mu_i +  b_t^{-1} \sum_{i = 0}^{t-1} \mu_i = a_t^{-1} \left(\sum_{i = t+1}^w \mu_i +   \sum_{i = 0}^{t-1} \frac{\mu_{i+1}b_{i+1}a_t}{a_i b_t}\right) \le  a_t^{-1} \left(\sum_{i = 1}^w \mu_i \right)
\end{align*}
The last inequality again uses the monotonicity of $a$ and $b$ and the fact $\mu$ sums to 1. The proof that $\Delta f(t) \le b_t^{-1}$ is essentially identical so we omit it. 

Finally, we turn to claim (d) of the lemma. We need three facts that we have already established: $\Delta f(s) < 0$ for all $s \neq t$; $f(s) \le 0$ for $s \le t$; and $f(s) \ge 0$ for $s > t$. Combining the first two facts, $\sum_{s \le t-1} |\Delta f(s)|$ telescopes and is bounded above by $|f(t)|$. Similarly, combining the second two facts, $\sum_{s \ge t+1} |\Delta f(s)| \le |f(t+1)|$. Since $|f(t+1)| + |f(t)| = f(t+1)-f(t) = |\Delta f(t)|$,
\begin{align*}
    \sum_{s} |\Delta f(s)| = \sum_{s \le t-1} |\Delta f(s)| + |\Delta f(t)| + \sum_{t+1 \le s} |\Delta f(s)| &\le |f(t)| + |\Delta f(t)| + |f(t+1)| \\
    &=  2|\Delta f(t)|\,.
\end{align*}
\QEDA \\

\section{Local limit theorems}\label{lemmas}
We collect several approximations which we use throughout.
\begin{lemma}[{De Moivre--Laplace~\cite[Theorem VII.3.1]{Fel68}}]\label{demoivre}
Let $p \in [\delta, 1-\delta]$ for a constant $\delta \in (0, 1/2)$, and write $q = 1 -p$.
Let $k = r - np$.
For any $r \in \{0, \dots, n\}$, it holds
\begin{equation*}
\binom{n}{r} p^r (1-p)^{n-r} = \left(1 + \cO{\frac 1n} \right) \frac{1}{\sqrt{2 \pi n p (1-p)}} \e^{-\frac{k^2}{2np(1-p)} + \cO{\frac kn } + \cO{\frac{k^3}{n^2}}}\,.
\end{equation*}
\end{lemma}
Specializing to $p = 1/2$ yields the following simplified bound.
\begin{lemma}[{Stirling's approximation for Binomial coefficients~\cite[Equation (5.43)]{Spe14}}]\label{stirling}
If $|r - n/2| = \oo(n^{-2/3})$, then
\begin{equation*}
\binom{n}{r} = \mo{1}\sqrt{\frac{2}{\pi n}} 2^n \e^{-2(r-n/2)^2/n}\,.
\end{equation*}
\end{lemma}
We also have a coarser estimate valid for all $r$.
\begin{lemma}[Gaussian tails for Binomial]\label{cramer}
In the same setting as \cref{demoivre}, for any $r \in \{0, \dots, n\}$,
\begin{equation*}
\binom{n}{r} p^r (1-p)^{n-r} = \cO{\frac{1}{\sqrt n}} \e^{-\frac{k^2}{4np(1-p)}}\,.
\end{equation*}
\end{lemma}
\begin{proof}
When $k = \oo(n^{2/3})$, this bound follows from \cref{demoivre}.
When $k = \cOm{n^{2/3}}$, we employ Hoeffding's bound~\cite{hoeffding1994probability}; letting $X \sim \Bin(n, p)$,
\begin{equation*}
\binom{n}{r} p^r (1-p)^{n-r} \leq \PP{|X - \E X| \geq k} \leq 2 \e^{- \frac{k^2}{2np(1-p)}} = \cO{\frac{1}{\sqrt n}}\e^{-\frac{k^2}{4np(1-p)}}\,,
\end{equation*}
where the last step uses $\e^{- k^2/n} = \cO{\frac{1}{\sqrt n}}$ if $k = \cOm{n^{2/3}}$.
\end{proof}

\begin{lemma}[{Gaussian tail for Hypergeometric distribution~\cite[Theorem 2, (i)]{hypergeo}}]\label{hyp_stirling}
Consider a hypergeometric random variable $X \sim H(w;pn,n)$ with $w$ trials, $n$ total population, and $np$ population successes. There is some universal constant $K$ such that for any $\lambda > 0$, 
\begin{equation*}
    \PP{\sqrt{w}|X-pw| = \lambda} \le \frac{K}{\sqrt{n}} \cE{-\frac{2\lambda^2}{1-w/n}- \left(\frac{1}{4}+\frac{1}{3}\originalleft(\frac{w}{n-w}\originalright)^3\right)\frac{\lambda^4}{n}}
\end{equation*}
In particular, if $n-w = \cT{n}$,
\begin{equation*}
    \binom{np}{k}\binom{n(1-p)}{w-k}\binom{n}{w}^{-1} = \cO{\frac{1}{\sqrt{w}}} e^{-\cT{\frac{1}{w}}\left(k-wp\right)^2}
\end{equation*}
\end{lemma}

\begin{lemma}[{Edgeworth Series for lattice sums \cite[Theorem 2]{petrov1964local}}]\label{edgeworth}
Let $\{X_i\}_{i \in \mathbb{N}}$ be independent identically distributed random variables with $X_1 \in \{-1,0,1\}$, $\E{X_1} = 0$, and $\E{X_1^2} = \sigma^2$. Denote by $\kappa_r$ the $r$-th cumulant of $X_1$, and denote by $\phi$ the density of standard unit Gaussian. Define the quantities
\[
    P_n(N) = \PP{\sum_{j=1}^n X_j = N}, \quad x = \frac{N}{\sigma \sqrt{n}}
\]
Then, there exists a collection of polynomials $\{q_{3\nu}(x)\}_{\nu \in \mathbb{N}}$, each of degree $3\nu$ with coefficients depending only on the moments of $X_1$ up to  order $\nu+2$ (inclusive), satisfying
\begin{equation*}
    \left| \sigma \sqrt{n} P_n(N) - \phi(x) - \sum_{\nu = 1}^{k-2} \frac{q_{3\nu}(x)\phi(x)}{n^{\nu/2}}\right| = \frac{\oo(1)}{\left(1 + |x|^k\right)n^{(k-2)/2}}
\end{equation*}
In particular, $q_{3\nu} = 0$ for any odd $\nu$, and the first non-zero $q$ is $q_6$ given by:
\begin{align*}
    q_{6}(x)\phi(x) &= -\left(\frac{1}{24}\lambda_4\phi^{(4)}(x) + \frac{1}{72}\lambda_3^2 \phi^{(6)}(x)\right)
\end{align*}
where $\lambda_r = \frac{\kappa_r}{\sigma^r}$ and $\phi^{(r)}$ is the $r$th derivative of the Gaussian density.
\end{lemma}

\begin{lemma}[{Exponential upper-tail for Poisson distribution \cite[Theorem 1]{po}} ]\label{poisson-tail}
If $S$ has Poisson distribution with mean $\lambda$, then
\[
    \PP{|S- \lambda| > x} \leq 2\e^{-\frac{x^2}{2(\lambda + x)}}\,.
\]
\end{lemma}

\bibliography{discrep}
\bibliographystyle{acm}

\end{document}